\documentclass[11pt,oneside,reqno]{amsart}

\usepackage[T1]{fontenc}
\usepackage{amssymb}
\usepackage{algpseudocode}
\usepackage{algorithm}
\usepackage{verbatim}
\usepackage{graphicx}
\usepackage{placeins}
\usepackage{listings}
\usepackage{xcolor}
\usepackage{subcaption}
\usepackage[noadjust]{cite}
\RequirePackage{dsfont} 
 \scrollmode
\allowdisplaybreaks
\usepackage{graphicx}
\newcommand{\TV}{\mathrm{TV}}
\usepackage{epstopdf}
\usepackage[hidelinks]{hyperref}
\usepackage{mathtools}

\hypersetup{
  pdftitle={Error and optimality of a randomized Milstein scheme under noisy information},
  pdfauthor={Paweł M. Morkisz, Paweł Przybyłowicz, Martyna Wiącek}
}

\usepackage{amsmath}
\usepackage{tikz}


\definecolor{codegreen}{rgb}{0,0.6,0}
\definecolor{codegray}{rgb}{0.5,0.5,0.5}
\definecolor{codepurple}{rgb}{0.58,0,0.82}
\definecolor{backcolour}{rgb}{0.95,0.95,0.92}

\lstdefinestyle{list_style}{
  backgroundcolor=\color{backcolour}, commentstyle=\color{codegreen},
  keywordstyle=\color{magenta},
  numberstyle=\tiny\color{codegray},
  stringstyle=\color{codepurple},
  basicstyle=\ttfamily\footnotesize,
  breakatwhitespace=false,         
  breaklines=true,                 
  captionpos=b,                    
  keepspaces=true,                 
  numbers=left,                    
  numbersep=5pt,                  
  showspaces=false,                
  showstringspaces=false,
  showtabs=false,                  
  tabsize=2
}

\newcommand{\xdasharrow}[2][->]{
\tikz[baseline=-\the\dimexpr\fontdimen22\textfont2\relax]{
\node[anchor=south,font=\scriptsize, inner ysep=1.5pt,outer xsep=2.2pt](x){#2};
\draw[shorten <=3.4pt,shorten >=3.4pt,dashed,#1](x.south west)--(x.south east);
}
}

\newcommand{\DEBUG}{}

\ifdefined\DEBUG%

  \def\rem#1{{\marginpar{\raggedright\scriptsize #1}}}
  \newcommand{\pmr}[1]{\rem{\color{blue}{$\bullet$ #1}}}
  \newcommand{\ppr}[1]{\rem{\color{red}{$\bullet$ #1}}}
 \else%

  \newcommand{\ppr}[1]{}
  \newcommand{\pmr}[1]{}
 \fi

\newcommand{\E}{{\mathbb E}}

\def\rho{\varrho_1}

\def\rd{\,{\mathrm d}}

\theoremstyle{plain}
\newtheorem{theorem}{Theorem}
\newtheorem{lemma}{Lemma}

\newtheorem{corollary}{Corollary}
\newtheorem{proposition}{Proposition}
\theoremstyle{definition}
\newtheorem{definition}{Definition}
\newtheorem{remark}{Remark}

\begin{document}

\title
[Randomized Milstein algorithm based on noisy information]
{Error of randomized Milstein scheme for scalar SDEs with noisy information about coefficients and Wiener process}

\author[P. M. Morkisz]{Pawe{\l } M. Morkisz}
\address{NVIDIA Corp. and AGH University of Krakow,
	Faculty of Applied Mathematics,
	Al. A.~Mickiewicza 30, 30-059 Krak\'ow, Poland}
\email{morkiszp@agh.edu.pl, corresponding author}

\author[P. Przybyłowicz]{Paweł Przybyłowicz}
\address{AGH University of Krakow,
Faculty of Applied Mathematics,
Al. A.~Mickiewicza 30, 30-059 Krak\'ow, Poland}
\email{pprzybyl@agh.edu.pl}

\author[M. Wi\k{a}cek]{Martyna Wi\k{a}cek}
\address{AGH University of Krakow,
Faculty of Applied Mathematics,
Al. A.~Mickiewicza 30, 30-059 Krak\'ow, Poland}
\email{martynawiacek@agh.edu.pl}

\begin{abstract}
We investigate the strong approximation of scalar stochastic differential equations
when the available standard information about the drift coefficient, the diffusion coefficient,
the derivative of the diffusion coefficient, and the observed Wiener path is corrupted by noise.
The precision of the drift, diffusion-information, and Wiener-path observations is described
by three nonnegative parameters \(\delta_1,\delta_2,\delta_3\), where \(\delta_2\) controls both
the noisy diffusion coefficient and the separate noisy derivative oracle required in the Milstein
correction. We analyze a randomized Milstein scheme based only on this noisy information
and prove, for \(r\ge2\), that its \(L^r\)-error is bounded by
\[
C\bigl(n^{-\min\{\gamma_1+\frac12,\gamma_2\}}+\delta_1+\delta_2+\delta_3\bigr),
\]
where \(n\) is the number of time steps and \(\gamma_1,\gamma_2\) are the temporal Hölder
exponents of the coefficients. We also prove a matching minimax lower bound in the
randomized standard-information model considered in the paper. In particular, the
Wiener-path contribution proportional to \(\delta_3\) is unavoidable, and the noisy randomized
Milstein scheme is minimax order-optimal.
\end{abstract}


\maketitle

\section{Introduction}
We investigate the strong approximation of solutions to scalar stochastic differential equations (SDEs)
\begin{equation}
\label{main_equation}
\left\{
\begin{array}{ll}
\rd X(t)=a(t,X(t))\,\rd t+b(t,X(t))\,\rd W(t), & t\in[0,T],\\[0.3em]
X(0)=\eta,
\end{array}
\right.
\end{equation}
where \(T>0\), \(W\) is a standard Wiener process, and \(\eta\) is the initial value.
Our goal is to approximate the terminal value \(X(T)\) in the strong sense, i.e. in \(L^r\)-norms, under the assumption that the available information about the coefficients and the driving path is inexact.

Classical numerical analysis of SDEs usually assumes exact access to point evaluations of the drift coefficient \(a\), the diffusion coefficient \(b\), and the increments of the driving Wiener process \(W\).
This idealization is often too strong. In computational finance, molecular dynamics, control, or data-driven modelling, coefficients may be obtained through calibration, statistical estimation, empirical measurements, or machine learning procedures. The resulting input data are therefore affected by modelling and numerical errors. Moreover, in practical simulation one often works not with an ideal Brownian path, but with a generated, discretized, filtered, or otherwise perturbed stochastic source. In this paper we formalize this situation through a smooth Markovian perturbation of the observed Wiener path, rather than through an arbitrary path-dependent perturbation model.

The analysis of algorithms under partial, noisy, or otherwise restricted information is one of the central themes of information-based complexity (IBC). For SDEs, sharp error estimates and optimality results under inexact coefficient information have been studied, among others, in \cite{PMPP17,PMPP19}. In this work we extend this viewpoint by incorporating an explicit perturbation model for the observed Wiener path. Recent work has introduced an IBC framework for randomized Euler approximation under noisy coefficient and Wiener-path information; see \cite{BaranekEtAl2026}. The present paper extends this setting to a randomized Milstein scheme. This requires a separate noisy information channel for \(b_y\) and leads to additional perturbation terms generated by the Milstein correction. Our main new lower-bound contribution is a sharp estimate linear in the Wiener-path noise level \(\delta_3\).

The numerical method considered in this paper is a randomized Milstein scheme under noisy information. The randomization is used in the drift term, as in randomized quadrature, and is particularly useful when the drift has limited temporal regularity. Since the Milstein correction involves
\[
L_1b(t,x)=b(t,x)b_y(t,x),
\]
our information model includes a separate noisy oracle for \(b_y\). This point is important: we do not differentiate the noisy diffusion coefficient \(\tilde b\), since the perturbation of \(b\) is only assumed to be measurable and may destroy differentiability. Instead, the noisy Milstein coefficient is constructed as
\[
\widetilde{L_1b}(t,x)=\tilde b(t,x)\,\widetilde{b_y}(t,x).
\]

The main contributions of the paper are as follows.
Firstly, we introduce an IBC framework for strong approximation of scalar SDEs in which the available data consist of noisy evaluations of \(a\), \(b\), \(b_y\), and a noisy version of the Wiener path. The corresponding precision levels are denoted by \(\delta_1,\delta_2,\delta_3\). The parameter \(\delta_3\) controls the perturbation of the observed Wiener process.

Then, we prove that the randomized Milstein scheme is stable with respect to the introduced noisy information. More precisely, the perturbation part of the error is bounded by
    \[
    C(\delta_1+\delta_2+\delta_3).
    \]
    Combining this stability estimate with the exact-information randomized Milstein rate yields the upper bound
    \[
    C\Bigl(n^{-\min\{\gamma_1+\frac12,\gamma_2\}}+\delta_1+\delta_2+\delta_3\Bigr),
    \]
    where \(n\) is the number of time steps and \(\gamma_1,\gamma_2\) describe the temporal regularity of the drift and diffusion coefficients.

Consequently, we prove that the dependence on \(\delta_3\) is unavoidable. The proof is a two-point lower-bound argument in the spirit of IBC, but with a statistical component: exact indistinguishability of information is replaced by near indistinguishability of observation laws in total variation. The latter is controlled by Pinsker's inequality and a Cameron--Martin relative-entropy computation.

Finally, we prove optimality, i.e., together with the known lower bounds for exact-information discretization and coefficient perturbations, the new lower bound for \(\delta_3\) yields minimax order-optimality, in the randomized
standard-information model considered below, up to multiplicative constants, for any $r\geq 2$:
    \[
e_N^{(r)}
\asymp
N^{-\min\{\gamma_1+\frac12,\gamma_2\}}
+\delta_1+\delta_2+\delta_3
    \]
    Thus, under the present assumptions, the proposed noisy randomized Milstein scheme attains
the minimax order in the considered information model.

We complement the theory with numerical experiments illustrating the perturbation response, fixed-noise error floors, balanced noise regimes, and the lower-bound mechanism behind the term proportional to the Wiener-path noise level \(\delta_3\).

The paper is organized as follows. Section~\ref{sec:preliminaries} introduces the input classes, the noisy information model, the randomized standard-information model, the minimal error, and the noisy randomized Milstein scheme. Section~\ref{sec:upper} proves the perturbation stability estimate and derives the upper error bound. Section~\ref{sec:lower} establishes the lower bounds, with particular emphasis on the Wiener-path noise contribution. Section~\ref{sec:numerical} presents numerical experiments.

\section{Preliminaries}\label{sec:preliminaries}

We write
\[
\mathbb N=\{1,2,\ldots\},
\qquad
\mathbb N_0=\{0,1,2,\ldots\}.
\]
We fix \(T>0\), \(K>0\), and regularity parameters
\(\gamma_1,\gamma_2\in(0,1]\).
Let \(W=\{W(t)\}_{t\ge0}\) be a standard one-dimensional Wiener process defined on a complete probability space
\((\Omega,\Sigma,\mathbb P)\).
Let \(\{\Sigma_t\}_{t\ge0}\) be a filtration satisfying the usual conditions such that \(W\) is a Wiener process with respect to \(\{\Sigma_t\}_{t\ge0}\), and set
\[
\Sigma_\infty:=\sigma\Bigl(\bigcup_{t\ge0}\Sigma_t\Bigr).
\]
For a real-valued random variable \(Y\) and \(r\in[1,\infty)\), we write
\[
\|Y\|_r:=\bigl(\mathbb E|Y|^r\bigr)^{1/r}.
\]

We assume throughout that the probability space is sufficiently rich to carry every finite auxiliary random vector used by the algorithms below, independently of \(\Sigma_\infty\).
Such a probability space exists by the {\L}omnicki--Ulam theorem; see, for example, \cite[p.~93]{Kallenberg}.
Thus no explicit enlargement or product construction of the underlying probability space will be needed.

We consider the scalar SDE
\begin{equation}
\label{eq:sde}
\left\{
\begin{array}{ll}
\rd X(t)=a(t,X(t))\,\rd t+b(t,X(t))\,\rd W(t), & t\in[0,T],\\[0.3em]
X(0)=\eta.
\end{array}
\right.
\end{equation}

\subsection{Admissible input classes}
\label{subsec:input-classes}

We first define the drift class.

\begin{definition}[Drift coefficients]
For \(\gamma_1\in(0,1]\), let \(\mathcal A_K^{\gamma_1}\) be the class of all functions
\(a:[0,T]\times\mathbb R\to\mathbb R\) such that
\begin{itemize}
    \item[(A1)] \(a\in C^{0,2}([0,T]\times\mathbb R)\),
    \item[(A2)] \(|a(0,0)|\le K\),
    \item[(A3)] \(|a(t,x)-a(t,y)|\le K|x-y|\),
    \item[(A4)] \(|a(t,x)-a(s,x)|\le K(1+|x|)|t-s|^{\gamma_1}\),
    \item[(A5)] \(|a_x(t,x)-a_x(t,y)|\le K|x-y|\),
    \item[(A6)] \(|a_x(t,x)-a_x(s,x)|\le K(1+|x|)|t-s|^{\gamma_1}\),
\end{itemize}
for all \(s,t\in[0,T]\) and \(x,y\in\mathbb R\), where \(a_x=\partial a/\partial x\).
\end{definition}

For a diffusion coefficient \(b\), we write
\[
b_y(t,x):=\frac{\partial b}{\partial x}(t,x),
\qquad
L_1b(t,x):=b(t,x)b_y(t,x).
\]

\begin{definition}[Diffusion coefficients]
For \(\gamma_2\in(0,1]\), let \(\mathcal B_K^{\gamma_2}\) be the class of all functions
\(b:[0,T]\times\mathbb R\to\mathbb R\) such that
\begin{itemize}
    \item[(B1)] \(b\in C^{0,2}([0,T]\times\mathbb R)\),
    \item[(B2)] \(|b(0,0)|\le K\),
    \item[(B3)] \(|b(t,x)-b(t,y)|\le K|x-y|\),
    \item[(B4)] \(|b(t,x)-b(s,x)|\le K(1+|x|)|t-s|^{\gamma_2}\),
    \item[(B5)] \(|b_y(t,x)|\le K\),
    \item[(B6)] \(|b_y(t,x)-b_y(t,y)|\le K|x-y|\),
    \item[(B7)] \(|L_1b(t,x)-L_1b(t,y)|\le K|x-y|\),
    \item[(B8)] \(|L_1b(t,x)-L_1b(s,x)|\le K(1+|x|)|t-s|^{\gamma_2}\),
\end{itemize}
for all \(s,t\in[0,T]\) and \(x,y\in\mathbb R\).
\end{definition}

For the initial condition, we set
\[
\mathcal J_K
:=
\bigl\{
\eta:\Omega\to\mathbb R
\ \big|\
\eta\text{ is }\Sigma_0\text{-measurable and }|\eta|\le K\text{ a.s.}
\bigr\}.
\]
The input class is
\begin{equation}
\label{eq:F-class}
\mathcal F(\gamma_1,\gamma_2,K)
:=
\mathcal A_K^{\gamma_1}\times\mathcal B_K^{\gamma_2}\times\mathcal J_K.
\end{equation}
For every \((a,b,\eta)\in\mathcal F(\gamma_1,\gamma_2,K)\), equation \eqref{eq:sde} admits a unique strong solution, denoted by
\[
X^{a,b,\eta}=\{X^{a,b,\eta}(t)\}_{t\in[0,T]}.
\]

\subsection{Noisy coefficient information}
\label{subsec:noisy-coefficients}

We consider precision parameters
\[
\delta_1,\delta_2,\delta_3\in[0,1],
\]
corresponding to perturbations of the drift coefficient \(a\), the diffusion information \((b,b_y)\), and the Wiener process \(W\), respectively.

We use the following classes of admissible perturbation functions:
\[
\mathcal K_{\mathrm{lin}}
:=
\Bigl\{
p:[0,T]\times\mathbb R\to\mathbb R
\ \Big|\
p\text{ is Borel measurable and } |p(t,x)|\le 1+|x|
\Bigr\},
\]
and
\[
\mathcal K_{\mathrm{bd}}
:=
\Bigl\{
q:[0,T]\times\mathbb R\to\mathbb R
\ \Big|\
q\text{ is Borel measurable and } |q(t,x)|\le 1
\Bigr\}.
\]

For fixed \(a\in\mathcal A_K^{\gamma_1}\) and \(b\in\mathcal B_K^{\gamma_2}\), define
\begin{equation}
\label{eq:Va-class}
V_a(\delta_1)
:=
\bigl\{
\tilde a:\ \tilde a=a+\delta_1p_a,\quad p_a\in\mathcal K_{\mathrm{lin}}
\bigr\},
\end{equation}
\begin{equation}
\label{eq:Vb-class}
V_b(\delta_2)
:=
\bigl\{
\tilde b:\ \tilde b=b+\delta_2p_b,\quad p_b\in\mathcal K_{\mathrm{lin}}
\bigr\},
\end{equation}
and
\begin{equation}
\label{eq:Vbby-class}
V_{b,b_y}(\delta_2)
:=
\bigl\{
\widetilde{b_y}:\ \widetilde{b_y}=b_y+\delta_2q_b,\quad q_b\in\mathcal K_{\mathrm{bd}}
\bigr\}.
\end{equation}
Thus \(V_a(\delta_1)\), \(V_b(\delta_2)\), and \(V_{b,b_y}(\delta_2)\) describe three separate noisy information channels: the drift coefficient, the diffusion coefficient, and the derivative information associated with the diffusion coefficient. The notation \(V_{b,b_y}\) emphasizes that this last class depends on the fixed diffusion coefficient \(b\) through its derivative \(b_y\).

The Milstein correction computed from noisy coefficient information is
\begin{equation}
\label{eq:L1b-tilde-def}
\widetilde{L_1b}(t,x):=\tilde b(t,x)\,\widetilde{b_y}(t,x).
\end{equation}

\begin{remark}
Elements of \(V_b(\delta_2)\) and \(V_{b,b_y}(\delta_2)\) are separate information channels, although they are controlled by the same precision parameter \(\delta_2\).
In particular, \(\widetilde{b_y}\) is not assumed to be the derivative of \(\tilde b\).
This distinction is necessary because the perturbation in \(\tilde b=b+\delta_2p_b\) is only assumed to be measurable and may destroy differentiability.
Evaluations of \(\tilde b\) and \(\widetilde{b_y}\) are counted separately in the information cost.
\end{remark}

\subsection{Noisy Wiener information}
\label{subsec:noisy-wiener}

For the perturbation of the Wiener process, we introduce
\[
\widetilde{\mathcal K}
:=
\Bigl\{
p:[0,T]\times\mathbb R\to\mathbb R
\ \Big|\
p\in C^{1,2}([0,T]\times\mathbb R),\ |p(0,0)|\le 1,
\]
\[
\hspace{9em}
\max\{|p_t(t,x)|,\ |p_x(t,x)|,\ |p_{xx}(t,x)|\}\le 1
\text{ for all }(t,x)\in[0,T]\times\mathbb R
\Bigr\}.
\]
Every \(p\in\widetilde{\mathcal K}\) satisfies
\begin{equation}
\label{eq:pW-growth}
|p(t,x)|\le 1+T+|x|,
\qquad (t,x)\in[0,T]\times\mathbb R.
\end{equation}

For the fixed Wiener process \(W\), the admissible noisy paths are
\begin{multline}
\label{eq:W-class}
\mathcal W(W,\delta_3)
:=
\Bigl\{
\tilde W:[0,T]\times\Omega\to\mathbb R
\ \Big|\
\exists p_W\in\widetilde{\mathcal K}\text{ such that}\\
\tilde W(t,\omega)=W(t,\omega)+\delta_3p_W\bigl(t,W(t,\omega)\bigr)
\text{ for all }t\in[0,T]
\Bigr\}.
\end{multline}
Whenever \(\tilde W\in\mathcal W(W,\delta_3)\), we write
\begin{equation}
\label{eq:W-tilde-Z}
\tilde W(t)=W(t)+\delta_3Z(t),
\qquad
Z(t):=p_W(t,W(t)).
\end{equation}

\begin{remark}
The perturbation of \(W\) is not arbitrary path-dependent noise. It is a Markovian perturbation of the form \(p_W(t,W(t))\), with \(p_W\in C^{1,2}\) and uniformly bounded derivatives. By It\^o's formula,
\[
\rd Z(t)
=
p_{W,x}(t,W(t))\,\rd W(t)
+
\left(p_{W,t}(t,W(t))+\frac12p_{W,xx}(t,W(t))\right)\rd t,
\]
so \(Z\) is a continuous semimartingale with uniformly controlled coefficients.
\end{remark}

\subsection{Randomized standard information and algorithms}
\label{subsec:adaptive-information}

We reserve \(N\in\mathbb N\) for the information budget. The letter \(n\in\mathbb N\) will later denote the number of time steps of the concrete Milstein scheme.

Fix \(i_1,i_2\in\mathbb N_0\). Let
\[
\boldsymbol\xi=(\xi_0,\ldots,\xi_{i_1-1})
\]
be an \([0,T]^{i_1}\)-valued random vector such that
\begin{equation}
\label{eq:xi-independent}
\sigma(\boldsymbol\xi)\quad\text{is independent of}\quad\Sigma_\infty.
\end{equation}
The law of \(\boldsymbol\xi\) is part of the information mapping and is fixed independently of the particular input.
Let
\[
t_0,\ldots,t_{i_1-1}\in[0,T],
\qquad
s_0,\ldots,s_{i_2-1}\in[0,T]
\]
be deterministic observation times, fixed in advance. We assume that the points within each of these two families are pairwise distinct.

For a noisy input
\[
(\tilde a,\tilde b,\widetilde{b_y},\tilde W,\eta),
\]
put
\[
\mathbf w
:=
\bigl(\tilde W(s_0),\ldots,\tilde W(s_{i_2-1})\bigr).
\]
The spatial evaluation points may be chosen adaptively from the auxiliary random vector
\(\boldsymbol\xi\), previously observed coefficient values, the vector \(\mathbf w\), and \(\eta\).
More precisely, for \(j=0,\ldots,i_1-1\), let
\[
\psi_j:[0,T]^{i_1}\times\mathbb R^{3j+i_2+1}\to\mathbb R^3
\]
be Borel measurable. With the convention that the vectors below are empty when \(j=0\), define
\[
\mathbf a_j
:=
\bigl(\tilde a(\xi_0,y_0),\ldots,\tilde a(\xi_{j-1},y_{j-1})\bigr),
\]
\[
\mathbf b_j
:=
\bigl(\tilde b(t_0,z_0),\ldots,\tilde b(t_{j-1},z_{j-1})\bigr),
\]
and
\[
\mathbf b_{y,j}
:=
\bigl(\widetilde{b_y}(t_0,u_0),\ldots,\widetilde{b_y}(t_{j-1},u_{j-1})\bigr).
\]
The next spatial arguments are generated recursively by
\begin{equation}
\label{eq:adaptive-spatial-points}
(y_j,z_j,u_j)
=
\psi_j\bigl(\boldsymbol\xi,\mathbf a_j,\mathbf b_j,\mathbf b_{y,j},\mathbf w,\eta\bigr),
\qquad j=0,\ldots,i_1-1.
\end{equation}

The resulting vector of noisy standard information is
\begin{equation}
\label{eq:information-vector}
\begin{split}
\mathcal N(\tilde a,\tilde b,\widetilde{b_y},\tilde W,\eta)
:=
\bigl[&
\tilde a(\xi_0,y_0),\ldots,\tilde a(\xi_{i_1-1},y_{i_1-1}),\\
&\tilde b(t_0,z_0),\ldots,\tilde b(t_{i_1-1},z_{i_1-1}),\\
&\widetilde{b_y}(t_0,u_0),\ldots,
  \widetilde{b_y}(t_{i_1-1},u_{i_1-1}),\\
&\tilde W(s_0),\ldots,\tilde W(s_{i_2-1}),\eta
\bigr].
\end{split}
\end{equation}
The obvious empty-block convention is used when \(i_1=0\) or \(i_2=0\).
Its information cost is
\begin{equation}
\label{eq:information-cost}
\operatorname{cost}(\mathcal N):=3i_1+i_2.
\end{equation}
The exact value of \(\eta\) and the auxiliary random vector \(\boldsymbol\xi\) are available
to the algorithm without cost. This means that \(\boldsymbol\xi\) may enter both the spatial
selection rules and the final output map, but it is not counted in the information cost.

An algorithm based on \(\mathcal N\) is of the form
\begin{equation}
\label{eq:general-algorithm}
\mathcal A(\tilde a,\tilde b,\widetilde{b_y},\tilde W,\eta)
=
\varphi\!\left(
\boldsymbol\xi,
\mathcal N(\tilde a,\tilde b,\widetilde{b_y},\tilde W,\eta)
\right),
\end{equation}
where
\[
\varphi:[0,T]^{i_1}\times\mathbb R^{3i_1+i_2+1}\to\mathbb R
\]
is Borel measurable.
For \(N\in\mathbb N\), let \(\Phi_N\) denote the class of all algorithms of the form \eqref{eq:general-algorithm} based on information mappings satisfying
\[
3i_1+i_2\le N.
\]

\begin{remark}
The observation times \(t_j\) and \(s_k\), as well as the law of \(\boldsymbol\xi\), are fixed in advance; adaptivity is allowed only in the spatial arguments of the coefficient evaluations. This is the standard-information framework used in \cite{PMPP19}, with the second diffusion block replaced by a separate noisy information channel for \(b_y\), and with noisy rather than exact observations of the Wiener path.
\end{remark}

\subsection{Error criteria}
\label{subsec:error-criteria}

Let \(r\in[1,\infty)\), let \((a,b,\eta)\in\mathcal F(\gamma_1,\gamma_2,K)\), and let
\(\delta_1,\delta_2,\delta_3\in[0,1]\).
For \(\mathcal A\in\Phi_N\), define the local worst-case error
\begin{multline}
\label{eq:error-single-input}
e^{(r)}\bigl(\mathcal A,a,b,\eta,\delta_1,\delta_2,\delta_3\bigr)
\\
:=
\sup_{\tilde a\in V_a(\delta_1)}
\sup_{\tilde b\in V_b(\delta_2)}
\sup_{\widetilde{b_y}\in V_{b,b_y}(\delta_2)}
\sup_{\tilde W\in\mathcal W(W,\delta_3)}
\left\|
X^{a,b,\eta}(T)
-
\mathcal A(\tilde a,\tilde b,\widetilde{b_y},\tilde W,\eta)
\right\|_r .
\end{multline}
The worst-case error over the input class is
\begin{equation}
\label{eq:error-class}
e^{(r)}\bigl(\mathcal A,\mathcal F(\gamma_1,\gamma_2,K),\delta_1,\delta_2,\delta_3\bigr)
:=
\sup_{(a,b,\eta)\in\mathcal F(\gamma_1,\gamma_2,K)}
e^{(r)}\bigl(\mathcal A,a,b,\eta,\delta_1,\delta_2,\delta_3\bigr).
\end{equation}
Finally, the \(N\)-th minimal error is
\begin{equation}
\label{eq:minimal-error}
e_N^{(r)}\bigl(\mathcal F(\gamma_1,\gamma_2,K),\delta_1,\delta_2,\delta_3\bigr)
:=
\inf_{\mathcal A\in\Phi_N}
e^{(r)}\bigl(\mathcal A,\mathcal F(\gamma_1,\gamma_2,K),\delta_1,\delta_2,\delta_3\bigr).
\end{equation}

\subsection{The randomized Milstein scheme under noisy information}
\label{subsec:randomized-milstein}

We now define the particular algorithm analyzed in the paper. Let \(n\in\mathbb N\),
\[
h=\frac{T}{n},
\qquad
t_i=ih,
\qquad
i=0,1,\ldots,n.
\]
Let \(\xi_0,\ldots,\xi_{n-1}\) be independent random variables such that \(\xi_i\) is uniformly distributed on \([t_i,t_{i+1}]\), and
\[
\sigma(\xi_0,\ldots,\xi_{n-1})
\quad\text{is independent of}\quad
\Sigma_\infty.
\]

For a process \(Y=\{Y(t)\}_{t\in[0,T]}\), set
\[
\Delta Y_i:=Y(t_{i+1})-Y(t_i),
\qquad
i=0,1,\ldots,n-1,
\]
and define
\begin{equation}
\label{eq:I-def}
I_i(Y,Y):=\frac12\bigl((\Delta Y_i)^2-h\bigr).
\end{equation}

The auxiliary exact-information randomized Milstein scheme is
\begin{equation}
\label{eq:exact-rm}
\left\{
\begin{array}{ll}
X^{RM}_{n,0}=\eta,\\[0.3em]
X^{RM}_{n,i+1}
=
X^{RM}_{n,i}
+
a(\xi_i,X^{RM}_{n,i})h
+
b(t_i,X^{RM}_{n,i})\Delta W_i
+
L_1b(t_i,X^{RM}_{n,i})I_i(W,W),
\end{array}
\right.
\end{equation}
for \(i=0,1,\ldots,n-1\).

Given noisy information
\[
(\tilde a,\tilde b,\widetilde{b_y},\tilde W,\eta),
\]
the noisy randomized Milstein scheme is
\begin{equation}
\label{eq:noisy-rm}
\left\{
\begin{array}{ll}
\widetilde X^{RM}_{n,0}=\eta,\\[0.3em]
\widetilde X^{RM}_{n,i+1}
=
\widetilde X^{RM}_{n,i}
+
\tilde a(\xi_i,\widetilde X^{RM}_{n,i})h
+
\tilde b(t_i,\widetilde X^{RM}_{n,i})\Delta\tilde W_i\\[0.3em]
\hspace{7em}
+
\widetilde{L_1b}(t_i,\widetilde X^{RM}_{n,i})I_i(\tilde W,\tilde W),
\end{array}
\right.
\end{equation}
for \(i=0,1,\ldots,n-1\), where
\[
\widetilde{L_1b}(t,x)=\tilde b(t,x)\widetilde{b_y}(t,x).
\]

\begin{remark}
For \(Y=W\), the quantity \(I_i(W,W)\) coincides with the scalar iterated Itô integral
\[
\int_{t_i}^{t_{i+1}}\int_{t_i}^{s}\rd W(u)\,\rd W(s)
=
\frac12\bigl((\Delta W_i)^2-h\bigr).
\]
For \(Y=\tilde W\), the quantity \(I_i(\tilde W,\tilde W)\) should not be interpreted as an Itô iterated integral with respect to the semimartingale \(\tilde W\). It is the Brownian-type Milstein correction computed by the algorithm from the corrupted increments, using the Brownian quadratic variation term \(h\).
\end{remark}

The associated output algorithm is
\begin{equation}
\label{eq:rm-output}
\mathcal A_n^{RM,\mathrm{noisy}}(\tilde a,\tilde b,\widetilde{b_y},\tilde W,\eta)
:=
\widetilde X^{RM}_{n,n}.
\end{equation}

The information cost of \(\mathcal A_n^{RM,\mathrm{noisy}}\) is at most \(4n+1\). Indeed, the algorithm uses:
\begin{itemize}
    \item \(n\) evaluations of \(\tilde a\),
    \item \(n\) evaluations of \(\tilde b\),
    \item \(n\) evaluations of \(\widetilde{b_y}\),
    \item \(n+1\) evaluations of \(\tilde W\) at \(t_0,\ldots,t_n\).
\end{itemize}
The value of \(\tilde b(t_i,\widetilde X^{RM}_{n,i})\) queried for the diffusion term is reused in the computation of \(\widetilde{L_1b}(t_i,\widetilde X^{RM}_{n,i})\).
In the notation of Section~\ref{subsec:adaptive-information}, the scheme corresponds to
\(i_1=n\), \(i_2=n+1\), and \(s_i=t_i\), with the spatial arguments generated recursively by the current approximation.
Consequently,
\begin{equation}
\label{eq:cost-rm}
\mathcal A_n^{RM,\mathrm{noisy}}\in\Phi_{4n+1}.
\end{equation}

\section{Error of the randomized Milstein scheme under noisy information}\label{sec:upper}
We investigate the error of the randomized Milstein scheme when only noisy standard information about \(a\), \(b\), \(b_y\), and \(W\) is available for the scalar SDE \eqref{eq:sde}.

\subsection{Error decomposition for the noisy randomized Milstein scheme}

Throughout this subsection, we fix \(n\in\mathbb N\), set
\[
h=\frac{T}{n},
\qquad
t_i=ih,
\qquad
i=0,1,\dots,n,
\]
and consider the auxiliary exact and noisy randomized Milstein schemes
\eqref{eq:exact-rm} and \eqref{eq:noisy-rm}.

For convenience, we write
\[
X_i:=X^{RM}_{n,i},
\qquad
\widetilde X_i:=\widetilde X^{RM}_{n,i},
\qquad
e_i:=X_i-\widetilde X_i,
\qquad i=0,1,\dots,n.
\]

We also introduce the enlarged filtration
\begin{equation}
\label{eq:G-filtration-continuous}
\mathcal G_t:=\Sigma_t\vee\sigma(\xi_0,\dots,\xi_{n-1}),
\qquad t\in[0,T],
\end{equation}
and write
\begin{equation}
\label{eq:G-filtration}
\mathcal G_i:=\mathcal G_{t_i},
\qquad i=0,1,\dots,n.
\end{equation}
By the independence assumption in \eqref{eq:xi-independent}, the process \(W\) is still a Wiener process with respect to \(\{\mathcal G_t\}_{t\in[0,T]}\). Since \(\eta\) is \(\Sigma_0\)-measurable, and since the variables \(\xi_0,\dots,\xi_{n-1}\) are available to the scheme as auxiliary randomization, \(X_i\) and \(\widetilde X_i\) are \(\mathcal G_i\)-measurable for every \(i\). Consequently,
\begin{equation}
\label{eq:martingale-differences-basic}
\mathbb E[\Delta W_i\mid \mathcal G_i]=0,
\qquad
\mathbb E[I_i(W,W)\mid \mathcal G_i]=0,
\qquad i=0,1,\dots,n-1.
\end{equation}

We begin with the growth estimates.

\begin{lemma}
\label{lem:elementary-growth}
There exists a constant \(C=C(T,K,\gamma_1,\gamma_2)\) such that for all
\((a,b,\eta)\in\mathcal F(\gamma_1,\gamma_2,K)\), all
\(\tilde a\in V_a(\delta_1)\), all
\(\tilde b\in V_b(\delta_2)\), all
\(\widetilde{b_y}\in V_{b,b_y}(\delta_2)\), and all
\((t,x)\in[0,T]\times\mathbb R\),
\begin{align}
|a(t,x)|+|b(t,x)|+|L_1b(t,x)| &\le C(1+|x|), \label{eq:growth-base}\\
|\tilde a(t,x)|+|\tilde b(t,x)|+|\widetilde{L_1b}(t,x)| &\le C(1+|x|), \label{eq:growth-noisy}\\
|a(t,x)-\tilde a(t,x)| &\le \delta_1(1+|x|), \label{eq:noise-a}\\
|b(t,x)-\tilde b(t,x)| &\le \delta_2(1+|x|), \label{eq:noise-b}\\
|L_1b(t,x)-\widetilde{L_1b}(t,x)| &\le C\,\delta_2(1+|x|). \label{eq:noise-L1b}
\end{align}
\end{lemma}

\begin{proof}
The bounds \eqref{eq:growth-base} follow from the defining assumptions on
\(\mathcal A_K^{\gamma_1}\), \(\mathcal B_K^{\gamma_2}\), and \(L_1b\).

Next, by the definition of \(V_a(\delta_1)\),
\[
\tilde a(t,x)=a(t,x)+\delta_1 p_a(t,x),
\qquad |p_a(t,x)|\le 1+|x|,
\]
which yields \eqref{eq:noise-a}. Combining this with \eqref{eq:growth-base}, we obtain
\eqref{eq:growth-noisy} for \(\tilde a\).

Similarly,
\[
\tilde b(t,x)=b(t,x)+\delta_2 p_b(t,x),
\qquad |p_b(t,x)|\le 1+|x|,
\]
which implies \eqref{eq:noise-b} and the linear growth bound for \(\tilde b\).

Moreover,
\[
\widetilde{b_y}(t,x)=b_y(t,x)+\delta_2 q_b(t,x),
\qquad |q_b(t,x)|\le 1,
\]
hence
\[
|\widetilde{b_y}(t,x)|\le |b_y(t,x)|+\delta_2 \le K+1.
\]
Therefore, using the linear growth of \(\tilde b\),
\[
|\widetilde{L_1b}(t,x)|
=
|\tilde b(t,x)\widetilde{b_y}(t,x)|
\le C(1+|x|),
\]
which proves the remaining part of \eqref{eq:growth-noisy}.

Finally,
\begin{align*}
|L_1b(t,x)-\widetilde{L_1b}(t,x)|
&=
|b(t,x)b_y(t,x)-\tilde b(t,x)\widetilde{b_y}(t,x)|\\
&\le
|b(t,x)-\tilde b(t,x)|\,|b_y(t,x)|
+
|\tilde b(t,x)|\,|b_y(t,x)-\widetilde{b_y}(t,x)|\\
&\le
K\,\delta_2(1+|x|)
+
C(1+|x|)\,\delta_2,
\end{align*}
which yields \eqref{eq:noise-L1b}.
\end{proof}

We now analyze the perturbation of Wiener increments.

\begin{lemma}
\label{lem:W-increment-decomposition}
Let \(\tilde W\in\mathcal W(W,\delta_3)\), and write
\[
\tilde W(t)=W(t)+\delta_3 Z(t),
\qquad
Z(t)=p_W(t,W(t)),
\qquad
p_W\in\widetilde{\mathcal K}.
\]
Then, for every \(i=0,1,\dots,n-1\),
\begin{align}
\Delta \tilde W_i &= \Delta W_i + \delta_3 \Delta Z_i, \label{eq:dW-tilde-decomp}\\
I_i(\tilde W,\tilde W) &= I_i(W,W)+\delta_3\,\Delta W_i\,\Delta Z_i+\frac{\delta_3^2}{2}(\Delta Z_i)^2. \label{eq:I-tilde-decomp}
\end{align}
\end{lemma}

\begin{proof}
Identity \eqref{eq:dW-tilde-decomp} is immediate from the definition of \(\tilde W\).
Using \eqref{eq:dW-tilde-decomp}, we get
\[
(\Delta \tilde W_i)^2
=
(\Delta W_i)^2 + 2\delta_3\,\Delta W_i\,\Delta Z_i + \delta_3^2 (\Delta Z_i)^2.
\]
Hence
\begin{align*}
I_i(\tilde W,\tilde W)
&=
\frac12\bigl((\Delta \tilde W_i)^2-h\bigr)\\
&=
\frac12\bigl((\Delta W_i)^2-h\bigr)
+
\delta_3\,\Delta W_i\,\Delta Z_i
+
\frac{\delta_3^2}{2}(\Delta Z_i)^2\\
&=
I_i(W,W)+\delta_3\,\Delta W_i\,\Delta Z_i+\frac{\delta_3^2}{2}(\Delta Z_i)^2,
\end{align*}
which proves \eqref{eq:I-tilde-decomp}.
\end{proof}

We now derive the exact recursion for the difference \(e_i=X_i-\widetilde X_i\).

\begin{proposition}
\label{prop:error-decomposition}
For every \(i=0,1,\dots,n-1\),

\begin{equation}
\label{eq:error-recursion}
e_{i+1} = e_i + \underbrace{S^{(a)}_i + S^{(b)}_i + S^{(L)}_i}_{\text{Stability terms}} + \underbrace{N^{(a)}_i + N^{(b)}_i + N^{(L)}_i}_{\text{Coeff.\ noise terms}} + \underbrace{N^{(W,1)}_i + N^{(W,2)}_i + N^{(W,3)}_i}_{\text{Wiener noise terms}},
\end{equation}
where the individual components are defined as follows
\begin{align*}
S_i^{(a)}
&=
\bigl(a(\xi_i,X_i)-a(\xi_i,\widetilde X_i)\bigr)h,\\
S_i^{(b)}
&=
\bigl(b(t_i,X_i)-b(t_i,\widetilde X_i)\bigr)\Delta W_i,\\
S_i^{(L)}
&=
\bigl(L_1b(t_i,X_i)-L_1b(t_i,\widetilde X_i)\bigr)I_i(W,W),
\end{align*}
\begin{align*}
N_i^{(a)}
&=
\bigl(a(\xi_i,\widetilde X_i)-\tilde a(\xi_i,\widetilde X_i)\bigr)h,\\
N_i^{(b)}
&=
\bigl(b(t_i,\widetilde X_i)-\tilde b(t_i,\widetilde X_i)\bigr)\Delta W_i,\\
N_i^{(L)}
&=
\bigl(L_1b(t_i,\widetilde X_i)-\widetilde{L_1b}(t_i,\widetilde X_i)\bigr)I_i(W,W),
\end{align*}
\begin{align*}
N_i^{(W,1)}
&=
-\delta_3\,\tilde b(t_i,\widetilde X_i)\Delta Z_i,\\
N_i^{(W,2)}
&=
-\delta_3\,\widetilde{L_1b}(t_i,\widetilde X_i)\Delta W_i\,\Delta Z_i,\\
N_i^{(W,3)}
&=
-\frac{\delta_3^2}{2}\,\widetilde{L_1b}(t_i,\widetilde X_i)(\Delta Z_i)^2.
\end{align*}
Consequently, for every \(k=1,\dots,n\),
\begin{equation}
\label{eq:error-sum}
e_k
=
\sum_{i=0}^{k-1}
\Big(
S_i^{(a)}+S_i^{(b)}+S_i^{(L)}
+
N_i^{(a)}+N_i^{(b)}+N_i^{(L)}
+
N_i^{(W,1)}+N_i^{(W,2)}+N_i^{(W,3)}
\Big).
\end{equation}
\end{proposition}

\begin{proof}
Subtracting \eqref{eq:noisy-rm} from \eqref{eq:exact-rm}, we obtain
\begin{align*}
e_{i+1}
&=
e_i
+
\bigl(a(\xi_i,X_i)-\tilde a(\xi_i,\widetilde X_i)\bigr)h\\
&\quad
+
\bigl(b(t_i,X_i)\Delta W_i-\tilde b(t_i,\widetilde X_i)\Delta\tilde W_i\bigr)\\
&\quad
+
\bigl(L_1b(t_i,X_i)I_i(W,W)-\widetilde{L_1b}(t_i,\widetilde X_i)I_i(\tilde W,\tilde W)\bigr).
\end{align*}
Using Lemma~\ref{lem:W-increment-decomposition}, we write
\[
\Delta\tilde W_i=\Delta W_i+\delta_3\Delta Z_i
\]
and
\[
I_i(\tilde W,\tilde W)=I_i(W,W)+\delta_3\Delta W_i\Delta Z_i+\frac{\delta_3^2}{2}(\Delta Z_i)^2.
\]
Substituting these identities and then adding and subtracting
\(a(\xi_i,\widetilde X_i)\), \(b(t_i,\widetilde X_i)\), and \(L_1b(t_i,\widetilde X_i)\),
we obtain exactly \eqref{eq:error-recursion}. Formula \eqref{eq:error-sum} follows by summation,
since \(e_0=0\).
\end{proof}

The remaining issue is to control the Wiener-noise terms
\(N_i^{(W,1)}\), \(N_i^{(W,2)}\), and \(N_i^{(W,3)}\).
For this purpose we study the process \(Z(t)=p_W(t,W(t))\) more carefully.

\begin{lemma}
\label{lem:Z-semi}
Let \(\tilde W\in\mathcal W(W,\delta_3)\), let \(p_W\in\widetilde{\mathcal K}\) be such that
\(Z(t)=p_W(t,W(t))\), and define
\begin{equation}
\label{eq:sigma-mu-def}
\sigma(t):=\partial_x p_W(t,W(t)),
\qquad
\mu(t):=\Bigl(\partial_t+\frac12\partial_{xx}\Bigr)p_W(t,W(t)).
\end{equation}
Then
\begin{equation}
\label{eq:sigma-mu-bounds}
|\sigma(t)|\le 1,
\qquad
|\mu(t)|\le \frac32,
\qquad t\in[0,T].
\end{equation}
Moreover,
\begin{equation}
\label{eq:Z-semimartingale}
\rd Z(t)=\sigma(t)\,\rd W(t)+\mu(t)\,\rd t.
\end{equation}

For every \(i=0,1,\dots,n-1\),
\begin{equation}
\label{eq:dZ-split}
\Delta Z_i = M_i^{(0)} + D_i^{(0)},
\end{equation}
where
\begin{align}
M_i^{(0)}
&:=
\int_{t_i}^{t_{i+1}} \sigma(s)\,\rd W(s), \label{eq:M0-def}\\
D_i^{(0)}
&:=
\int_{t_i}^{t_{i+1}} \mu(s)\,\rd s. \label{eq:D0-def}
\end{align}

For every \(i=0,1,\dots,n-1\) and \(t\in[t_i,t_{i+1}]\), let
\[
U_i(t):=Z(t)-Z(t_i),
\qquad
V_i(t):=W(t)-W(t_i).
\]
Then
\begin{equation}
\label{eq:Ui-sde}
U_i(t)=\int_{t_i}^t \sigma(s)\,\rd W(s)+\int_{t_i}^t \mu(s)\,\rd s.
\end{equation}
In addition,
\begin{equation}
\label{eq:Ui-square}
U_i(t)^2
=
2\int_{t_i}^t U_i(s)\sigma(s)\,\rd W(s)
+
\int_{t_i}^t \Bigl(2U_i(s)\mu(s)+\sigma(s)^2\Bigr)\,\rd s,
\end{equation}
and
\begin{equation}
\label{eq:ViUi-product}
V_i(t)U_i(t)
=
\int_{t_i}^t \Bigl(U_i(s)+V_i(s)\sigma(s)\Bigr)\,\rd W(s)
+
\int_{t_i}^t \Bigl(V_i(s)\mu(s)+\sigma(s)\Bigr)\,\rd s.
\end{equation}

Consequently,
\begin{equation}
\label{eq:DeltaZ-square-MD}
(\Delta Z_i)^2 = M_i^{(2)} + D_i^{(2)},
\end{equation}
where
\begin{align}
M_i^{(2)}
&:=
2\int_{t_i}^{t_{i+1}} U_i(s)\sigma(s)\,\rd W(s), \label{eq:M2-def}\\
D_i^{(2)}
&:=
\int_{t_i}^{t_{i+1}} \Bigl(2U_i(s)\mu(s)+\sigma(s)^2\Bigr)\,\rd s, \label{eq:D2-def}
\end{align}
and
\begin{equation}
\label{eq:DeltaWDeltaZ-MD}
\Delta W_i\,\Delta Z_i = M_i^{(1)} + D_i^{(1)},
\end{equation}
where
\begin{align}
M_i^{(1)}
&:=
\int_{t_i}^{t_{i+1}} \Bigl(U_i(s)+V_i(s)\sigma(s)\Bigr)\,\rd W(s), \label{eq:M1-def}\\
D_i^{(1)}
&:=
\int_{t_i}^{t_{i+1}} \Bigl(V_i(s)\mu(s)+\sigma(s)\Bigr)\,\rd s. \label{eq:D1-def}
\end{align}
\end{lemma}

\begin{proof}
Identity \eqref{eq:Z-semimartingale} is a direct consequence of It\^o's formula applied to
\(p_W(t,W(t))\). Since \(p_W\in\widetilde{\mathcal K}\), we have
\[
|\partial_x p_W(t,x)|\le 1,
\qquad
|\partial_t p_W(t,x)|\le 1,
\qquad
|\partial_{xx}p_W(t,x)|\le 1,
\]
hence \eqref{eq:sigma-mu-bounds} follows immediately.

Formula \eqref{eq:dZ-split} is obtained by integrating \eqref{eq:Z-semimartingale} over \([t_i,t_{i+1}]\), and \eqref{eq:Ui-sde} follows by integrating \eqref{eq:Z-semimartingale} from \(t_i\) to \(t\).

Next, \eqref{eq:Ui-square} follows from It\^o's formula applied to the function \(x\mapsto x^2\)
and the semimartingale \(U_i\):
\[
\rd(U_i(t)^2)=2U_i(t)\,\rd U_i(t)+\rd\langle U_i\rangle_t.
\]
Since
\[
\rd U_i(t)=\sigma(t)\,\rd W(t)+\mu(t)\,\rd t,
\qquad
\rd\langle U_i\rangle_t=\sigma(t)^2\,\rd t,
\]
we obtain \eqref{eq:Ui-square}.

Similarly, \eqref{eq:ViUi-product} follows from the product rule
\[
\rd(V_i(t)U_i(t))=V_i(t)\,\rd U_i(t)+U_i(t)\,\rd V_i(t)+\rd\langle V_i,U_i\rangle_t.
\]
Here
\[
\rd V_i(t)=\rd W(t),
\qquad
\rd\langle V_i,U_i\rangle_t=\sigma(t)\,\rd t.
\]
This yields \eqref{eq:ViUi-product}.

Finally, \eqref{eq:DeltaZ-square-MD}--\eqref{eq:D1-def} are obtained by evaluating
\eqref{eq:Ui-square} and \eqref{eq:ViUi-product} at \(t=t_{i+1}\).
\end{proof}

The next lemma provides the conditional moment bounds needed in the stability analysis.

\begin{lemma}
\label{lem:Z-conditional-moments}
Let \(r\ge 2\). Then there exists a constant \(C_r<\infty\), depending only on
\(r\) and \(T\), such that for every \(i=0,1,\dots,n-1\),
\begin{align}
\mathbb E\Big[\sup_{t_i\le t\le t_{i+1}} |U_i(t)|^r \,\Big|\, \mathcal G_i\Big]
&\le C_r h^{r/2}, \label{eq:Ui-sup-moment}\\
\mathbb E\Big[\sup_{t_i\le t\le t_{i+1}} |V_i(t)|^r \,\Big|\, \mathcal G_i\Big]
&\le C_r h^{r/2}, \label{eq:Vi-sup-moment}
\end{align}
and
\begin{align}
\mathbb E\big[|\Delta Z_i|^r \mid \mathcal G_i\big] &\le C_r h^{r/2}, \label{eq:DeltaZ-cond}\\
\mathbb E\big[|M_i^{(1)}|^r \mid \mathcal G_i\big]
+
\mathbb E\big[|D_i^{(1)}|^r \mid \mathcal G_i\big]
&\le C_r h^r, \label{eq:M1D1-cond}\\
\mathbb E\big[|M_i^{(2)}|^r \mid \mathcal G_i\big]
+
\mathbb E\big[|D_i^{(2)}|^r \mid \mathcal G_i\big]
&\le C_r h^r. \label{eq:M2D2-cond}
\end{align}
\end{lemma}
\begin{proof}
By \eqref{eq:Ui-sde}, boundedness of \(\sigma\) and \(\mu\), the conditional
Burkholder--Davis--Gundy inequality, and Jensen's inequality,
\begin{align*}
\mathbb E\Big[\sup_{t_i\le t\le t_{i+1}} |U_i(t)|^r \,\Big|\, \mathcal G_i\Big]
&\le
C_r\,
\mathbb E\Big[\Big(\int_{t_i}^{t_{i+1}} |\sigma(s)|^2\,\rd s\Big)^{r/2}\,\Big|\,\mathcal G_i\Big]
+
C_r\Big(\int_{t_i}^{t_{i+1}} |\mu(s)|\,\rd s\Big)^r\\
&\le C_r h^{r/2},
\end{align*}
which proves \eqref{eq:Ui-sup-moment}. Estimate \eqref{eq:Vi-sup-moment} is the standard
conditional Brownian increment bound. Since \(\Delta Z_i=U_i(t_{i+1})\), \eqref{eq:DeltaZ-cond}
follows immediately from \eqref{eq:Ui-sup-moment}.

Next, using \eqref{eq:M1-def}, conditional BDG, \eqref{eq:Ui-sup-moment}, and \eqref{eq:Vi-sup-moment},
we obtain
\begin{align*}
\mathbb E\big[|M_i^{(1)}|^r \mid \mathcal G_i\big]
&\le
C_r\,
\mathbb E\Big[
\Big(
\int_{t_i}^{t_{i+1}}
|U_i(s)+V_i(s)\sigma(s)|^2\,\rd s
\Big)^{r/2}
\Bigm|
\mathcal G_i
\Big]\\
&\le
C_r\,h^{\frac r2-1}
\int_{t_i}^{t_{i+1}}
\mathbb E\Big[
|U_i(s)|^r+|V_i(s)|^r
\Bigm|
\mathcal G_i
\Big]\,\rd s\\
&\le C_r h^r.
\end{align*}
For \(D_i^{(1)}\), by \eqref{eq:D1-def}, \eqref{eq:sigma-mu-bounds}, Hölder's inequality,
and \eqref{eq:Vi-sup-moment},
\begin{align*}
\mathbb E\big[|D_i^{(1)}|^r \mid \mathcal G_i\big]
&\le
C_r\,h^{r-1}
\int_{t_i}^{t_{i+1}}
\mathbb E\Big[
1+|V_i(s)|^r
\Bigm|
\mathcal G_i
\Big]\,\rd s\\
&\le C_r h^r.
\end{align*}
This proves \eqref{eq:M1D1-cond}. The proof of \eqref{eq:M2D2-cond} is analogous,
using \eqref{eq:M2-def}, \eqref{eq:D2-def}, and \eqref{eq:Ui-sup-moment}.
\end{proof}

The following weighted maximal inequalities are the key tool for the treatment of the Wiener-noise terms.
They will be applied with \(\mathcal G_i\)-measurable weights generated by the scheme.

\begin{lemma}
\label{lem:weighted-increment-estimates}
Let \(r\ge 2\), and let \(\tilde W\in\mathcal W(W,\delta_3)\) with associated process \(Z\) as in \eqref{eq:W-tilde-Z}. Then there exists a constant \(C_r<\infty\), depending only on \(r\) and \(T\),
such that for every sequence \(\{Y_i\}_{i=0}^{n-1}\) of \(\mathcal G_i\)-measurable random variables,
the following estimates hold:
\begin{align}
\mathbb E\max_{1\le k\le n}
\left|
\sum_{i=0}^{k-1} Y_i\,\Delta W_i
\right|^r
&\le
C_r \sum_{i=0}^{n-1} h\,\mathbb E|Y_i|^r, \label{eq:weighted-dW}\\
\mathbb E\max_{1\le k\le n}
\left|
\sum_{i=0}^{k-1} Y_i\,I_i(W,W)
\right|^r
&\le
C_r \sum_{i=0}^{n-1} h\,\mathbb E|Y_i|^r, \label{eq:weighted-I}\\
\mathbb E\max_{1\le k\le n}
\left|
\sum_{i=0}^{k-1} Y_i\,\Delta Z_i
\right|^r
&\le
C_r \sum_{i=0}^{n-1} h\,\mathbb E|Y_i|^r, \label{eq:weighted-dZ}\\
\mathbb E\max_{1\le k\le n}
\left|
\sum_{i=0}^{k-1} Y_i\,\Delta W_i\,\Delta Z_i
\right|^r
&\le
C_r \sum_{i=0}^{n-1} h\,\mathbb E|Y_i|^r, \label{eq:weighted-dWdZ}\\
\mathbb E\max_{1\le k\le n}
\left|
\sum_{i=0}^{k-1} Y_i\,(\Delta Z_i)^2
\right|^r
&\le
C_r \sum_{i=0}^{n-1} h\,\mathbb E|Y_i|^r. \label{eq:weighted-dZ2}
\end{align}
\end{lemma}

\begin{proof}
We write \(\mathbf 1_i(s):=\mathbf 1_{[t_i,t_{i+1})}(s)\).

To prove \eqref{eq:weighted-dW}, we set
\[
H_W(s):=\sum_{i=0}^{n-1} Y_i\,\mathbf 1_i(s),
\qquad s\in[0,T].
\]
Since \(Y_i\) is \(\mathcal G_i\)-measurable, the process \(H_W\) is predictable with respect to
the filtration \(\{\mathcal G_t\}_{t\in[0,T]}\), where
\[
\mathcal G_t:=\sigma(\eta,\xi_0,\dots,\xi_{n-1})\vee\Sigma_t.
\]
Moreover,
\[
\sum_{i=0}^{k-1} Y_i\,\Delta W_i
=
\int_0^{t_k} H_W(s)\,\rd W(s),
\qquad k=1,\dots,n.
\]
Hence, by the Burkholder--Davis--Gundy inequality,
\begin{align*}
\mathbb E\max_{1\le k\le n}\left|\sum_{i=0}^{k-1}Y_i\,\Delta W_i\right|^r
&\le
\mathbb E\sup_{0\le t\le T}\left|\int_0^t H_W(s)\,\rd W(s)\right|^r\\
&\le
C_r\,\mathbb E\left(\int_0^T |H_W(s)|^2\,\rd s\right)^{r/2}\\
&=
C_r\,\mathbb E\left(\sum_{i=0}^{n-1} h\,|Y_i|^2\right)^{r/2}.
\end{align*}
Since \(r/2\ge 1\), Jensen's inequality gives
\[
\left(\sum_{i=0}^{n-1} h\,|Y_i|^2\right)^{r/2}
\le
T^{\frac r2-1}\sum_{i=0}^{n-1} h\,|Y_i|^r.
\]
Therefore,
\[
\mathbb E\max_{1\le k\le n}\left|\sum_{i=0}^{k-1}Y_i\,\Delta W_i\right|^r
\le
C_r\sum_{i=0}^{n-1} h\,\mathbb E|Y_i|^r.
\]

To prove \eqref{eq:weighted-I}, we recall that
\[
I_i(W,W)=\int_{t_i}^{t_{i+1}} V_i(s)\,\rd W(s),
\qquad
V_i(s):=W(s)-W(t_i).
\]
Hence
\[
\sum_{i=0}^{k-1}Y_i\,I_i(W,W)
=
\int_0^{t_k} H_I(s)\,\rd W(s),
\qquad
H_I(s):=\sum_{i=0}^{n-1} Y_i\,V_i(s)\,\mathbf 1_i(s).
\]
Again, by Burkholder--Davis--Gundy,
\[
\mathbb E\max_{1\le k\le n}\left|\sum_{i=0}^{k-1}Y_i\,I_i(W,W)\right|^r
\le
C_r\,\mathbb E\left(\sum_{i=0}^{n-1}|Y_i|^2\int_{t_i}^{t_{i+1}} |V_i(s)|^2\,\rd s\right)^{r/2}.
\]
Using \((\sum a_i)^{r/2}\le n^{r/2-1}\sum a_i^{r/2}\) for \(a_i\ge0\), we obtain
\[
\mathbb E\max_{1\le k\le n}\left|\sum_{i=0}^{k-1}Y_i\,I_i(W,W)\right|^r
\le
C_r\,n^{\frac r2-1}\sum_{i=0}^{n-1}
\mathbb E\left[
|Y_i|^r
\left(\int_{t_i}^{t_{i+1}} |V_i(s)|^2\,\rd s\right)^{r/2}
\right].
\]
Since \(Y_i\) is \(\mathcal G_i\)-measurable,
\[
\mathbb E\left[
|Y_i|^r
\left(\int_{t_i}^{t_{i+1}} |V_i(s)|^2\,\rd s\right)^{r/2}
\right]
=
\mathbb E\left[
|Y_i|^r
\,
\mathbb E\left[
\left(\int_{t_i}^{t_{i+1}} |V_i(s)|^2\,\rd s\right)^{r/2}
\Bigm|
\mathcal G_i
\right]
\right].
\]
By Jensen's inequality and \eqref{eq:Vi-sup-moment},
\begin{align*}
\mathbb E\left[
\left(\int_{t_i}^{t_{i+1}} |V_i(s)|^2\,\rd s\right)^{r/2}
\Bigm|
\mathcal G_i
\right]
&\le
h^{\frac r2-1}\int_{t_i}^{t_{i+1}}
\mathbb E\bigl[|V_i(s)|^r\mid\mathcal G_i\bigr]\,\rd s\\
&\le
h^{\frac r2-1}\int_{t_i}^{t_{i+1}}
\mathbb E\Big[\sup_{t_i\le u\le t_{i+1}}|V_i(u)|^r\Bigm|\mathcal G_i\Big]\,\rd s\\
&\le C_r h^r.
\end{align*}
Therefore,
\[
\mathbb E\max_{1\le k\le n}\left|\sum_{i=0}^{k-1}Y_i\,I_i(W,W)\right|^r
\le
C_r\,n^{\frac r2-1} h^r \sum_{i=0}^{n-1}\mathbb E|Y_i|^r.
\]
Since \(n^{\frac r2-1}h^r = T^{\frac r2-1}h^{1+\frac r2}\le C(T,r)\,h\), we get
\[
\mathbb E\max_{1\le k\le n}\left|\sum_{i=0}^{k-1}Y_i\,I_i(W,W)\right|^r
\le
C_r\sum_{i=0}^{n-1} h\,\mathbb E|Y_i|^r.
\]

For \eqref{eq:weighted-dZ}, we use \eqref{eq:dZ-split}, and we write
\[
\sum_{i=0}^{k-1} Y_i\,\Delta Z_i
=
\sum_{i=0}^{k-1} Y_i\,M_i^{(0)}
+
\sum_{i=0}^{k-1} Y_i\,D_i^{(0)}.
\]
For the martingale part we use
\[
\sum_{i=0}^{k-1} Y_i\,M_i^{(0)}
=
\int_0^{t_k} H_Z(s)\,\rd W(s),
\qquad
H_Z(s):=\sum_{i=0}^{n-1} Y_i\,\sigma(s)\,\mathbf 1_i(s).
\]
Since \(|\sigma(s)|\le 1\), the proof of \eqref{eq:weighted-dW} yields
\[
\mathbb E\max_{1\le k\le n}\left|\sum_{i=0}^{k-1} Y_i\,M_i^{(0)}\right|^r
\le
C_r\sum_{i=0}^{n-1} h\,\mathbb E|Y_i|^r.
\]
For the drift part, since \(|D_i^{(0)}|\le \int_{t_i}^{t_{i+1}} |\mu(s)|\,\rd s\le C h\), we have
\begin{align*}
\mathbb E\max_{1\le k\le n}\left|\sum_{i=0}^{k-1} Y_i\,D_i^{(0)}\right|^r
&\le
\mathbb E\left(\sum_{i=0}^{n-1}|Y_i|\,|D_i^{(0)}|\right)^r\\
&\le
n^{r-1}\sum_{i=0}^{n-1}\mathbb E\bigl[|Y_i|^r|D_i^{(0)}|^r\bigr]\\
&\le
C_r\,n^{r-1}h^r \sum_{i=0}^{n-1}\mathbb E|Y_i|^r\\
&=
C_r\,T^{r-1}\sum_{i=0}^{n-1} h\,\mathbb E|Y_i|^r.
\end{align*}
This proves \eqref{eq:weighted-dZ}.

By \eqref{eq:DeltaWDeltaZ-MD},
\[
\sum_{i=0}^{k-1} Y_i\,\Delta W_i\,\Delta Z_i
=
\sum_{i=0}^{k-1} Y_i\,M_i^{(1)}
+
\sum_{i=0}^{k-1} Y_i\,D_i^{(1)}.
\]
For the martingale part, define
\[
H_{12}(s):=\sum_{i=0}^{n-1}Y_i\bigl(U_i(s)+V_i(s)\sigma(s)\bigr)\mathbf 1_i(s).
\]
Then
\[
\sum_{i=0}^{k-1} Y_i\,M_i^{(1)}
=
\int_0^{t_k} H_{12}(s)\,\rd W(s).
\]
By Burkholder--Davis--Gundy,
\[
\mathbb E\max_{1\le k\le n}\left|\sum_{i=0}^{k-1} Y_i\,M_i^{(1)}\right|^r
\le
C_r\,
\mathbb E\left(
\sum_{i=0}^{n-1}|Y_i|^2
\int_{t_i}^{t_{i+1}} |U_i(s)+V_i(s)\sigma(s)|^2\,\rd s
\right)^{r/2}.
\]
As above,
\[
(\sum a_i)^{r/2}\le n^{r/2-1}\sum a_i^{r/2},
\]
and, conditionally on \(\mathcal G_i\),
\begin{align*}
\mathbb E\left[
\left(
\int_{t_i}^{t_{i+1}} |U_i(s)+V_i(s)\sigma(s)|^2\,\rd s
\right)^{r/2}
\Bigm|
\mathcal G_i
\right]
&\le
C_r h^{\frac r2-1}
\int_{t_i}^{t_{i+1}}
\mathbb E\bigl[|U_i(s)|^r+|V_i(s)|^r\mid\mathcal G_i\bigr]\,\rd s\\
&\le
C_r h^r
\end{align*}
by \eqref{eq:Ui-sup-moment}, \eqref{eq:Vi-sup-moment}, and \(|\sigma|\le 1\).
Thus
\[
\mathbb E\max_{1\le k\le n}\left|\sum_{i=0}^{k-1} Y_i\,M_i^{(1)}\right|^r
\le
C_r\sum_{i=0}^{n-1} h\,\mathbb E|Y_i|^r.
\]
For the drift part, \eqref{eq:M1D1-cond} implies
\[
\mathbb E\bigl[|D_i^{(1)}|^r\mid\mathcal G_i\bigr]\le C_r h^r.
\]
Therefore,
\begin{align*}
\mathbb E\max_{1\le k\le n}\left|\sum_{i=0}^{k-1} Y_i\,D_i^{(1)}\right|^r
&\le
n^{r-1}\sum_{i=0}^{n-1}
\mathbb E\bigl[|Y_i|^r|D_i^{(1)}|^r\bigr]\\
&=
n^{r-1}\sum_{i=0}^{n-1}
\mathbb E\Bigl[
|Y_i|^r\,\mathbb E[|D_i^{(1)}|^r\mid\mathcal G_i]
\Bigr]\\
&\le
C_r\,n^{r-1}h^r\sum_{i=0}^{n-1}\mathbb E|Y_i|^r
\le
C_r\sum_{i=0}^{n-1} h\,\mathbb E|Y_i|^r.
\end{align*}
This proves \eqref{eq:weighted-dWdZ}.

Finally, to prove \eqref{eq:weighted-dZ2}, we have from
\eqref{eq:DeltaZ-square-MD},
\[
\sum_{i=0}^{k-1} Y_i\,(\Delta Z_i)^2
=
\sum_{i=0}^{k-1} Y_i\,M_i^{(2)}
+
\sum_{i=0}^{k-1} Y_i\,D_i^{(2)}.
\]
For the martingale part we define
\[
H_{22}(s):=2\sum_{i=0}^{n-1}Y_i\,U_i(s)\sigma(s)\,\mathbf 1_i(s),
\]
so that
\[
\sum_{i=0}^{k-1} Y_i\,M_i^{(2)}
=
\int_0^{t_k} H_{22}(s)\,\rd W(s).
\]
Applying Burkholder--Davis--Gundy and arguing exactly as above, we obtain
\[
\mathbb E\max_{1\le k\le n}\left|\sum_{i=0}^{k-1} Y_i\,M_i^{(2)}\right|^r
\le
C_r\sum_{i=0}^{n-1} h\,\mathbb E|Y_i|^r,
\]
because
\[
\mathbb E\left[
\left(
\int_{t_i}^{t_{i+1}} |U_i(s)\sigma(s)|^2\,\rd s
\right)^{r/2}
\Bigm|
\mathcal G_i
\right]
\le
C_r h^r
\]
follows from \eqref{eq:Ui-sup-moment} and \(|\sigma|\le 1\).
For the drift part we use \eqref{eq:M2D2-cond} exactly as in the proof of
\eqref{eq:weighted-dWdZ} and get
\[
\mathbb E\max_{1\le k\le n}\left|\sum_{i=0}^{k-1} Y_i\,D_i^{(2)}\right|^r
\le
C_r\sum_{i=0}^{n-1} h\,\mathbb E|Y_i|^r.
\]
This proves \eqref{eq:weighted-dZ2}.
\end{proof}

We next derive a uniform moment bound for the auxiliary exact-information scheme.

\begin{proposition}
\label{prop:exact-scheme-moments}
Let \(r\ge 2\). Then there exists a constant \(C_r<\infty\), depending only on
\(r,T,K,\gamma_1,\gamma_2\), such that
\begin{equation}
\label{eq:exact-scheme-max-moment-bound}
\sup_{(a,b,\eta)\in\mathcal F(\gamma_1,\gamma_2,K)}
\sup_{n\in\mathbb N}
\mathbb E\max_{0\le i\le n}|X_i|^r
\le C_r.
\end{equation}
Consequently,
\begin{equation}
\label{eq:exact-scheme-moment-bound}
\sup_{(a,b,\eta)\in\mathcal F(\gamma_1,\gamma_2,K)}
\sup_{n\in\mathbb N}
\max_{0\le i\le n}\mathbb E|X_i|^r
\le C_r.
\end{equation}
\end{proposition}

\begin{proof}
We fix \(n\in\mathbb N\) and define
\[
M_k:=\mathbb E\max_{0\le j\le k}|X_j|^r,
\qquad k=0,1,\dots,n.
\]
By summing the recursion \eqref{eq:exact-rm}, we obtain
\[
X_k
=
\eta
+
\sum_{i=0}^{k-1} a(\xi_i,X_i)\,h
+
\sum_{i=0}^{k-1} b(t_i,X_i)\,\Delta W_i
+
\sum_{i=0}^{k-1} L_1b(t_i,X_i)\,I_i(W,W).
\]
Hence, using \((x_1+\cdots+x_4)^r\le C_r\sum_{\ell=1}^4 |x_\ell|^r\),
\begin{align*}
M_k
&\le
C_r\,\mathbb E|\eta|^r
+
C_r\,\mathbb E\Big(
\sum_{i=0}^{k-1} |a(\xi_i,X_i)|\,h
\Big)^r\\
&\quad
+
C_r\,\mathbb E\max_{1\le m\le k}
\left|
\sum_{i=0}^{m-1} b(t_i,X_i)\,\Delta W_i
\right|^r
+
C_r\,\mathbb E\max_{1\le m\le k}
\left|
\sum_{i=0}^{m-1} L_1b(t_i,X_i)\,I_i(W,W)
\right|^r.
\end{align*}
Since \(|\eta|\le K\) a.s., Lemma~\ref{lem:elementary-growth} yields
\[
\mathbb E|\eta|^r \le K^r.
\]
For the drift term, using \eqref{eq:growth-base} and Jensen's inequality,
\begin{align*}
\mathbb E\Big(
\sum_{i=0}^{k-1} |a(\xi_i,X_i)|\,h
\Big)^r
&\le
T^{r-1}\sum_{i=0}^{k-1} h\,\mathbb E|a(\xi_i,X_i)|^r\\
&\le
C_r\sum_{i=0}^{k-1} h\,\bigl(1+\mathbb E|X_i|^r\bigr)
\le
C_r + C_r\sum_{i=0}^{k-1} h\,M_i.
\end{align*}
For the diffusion and Milstein terms, Lemma~\ref{lem:weighted-increment-estimates} and
\eqref{eq:growth-base} imply
\begin{align*}
\mathbb E\max_{1\le m\le k}
\left|
\sum_{i=0}^{m-1} b(t_i,X_i)\,\Delta W_i
\right|^r
&\le
C_r\sum_{i=0}^{k-1} h\,\mathbb E|b(t_i,X_i)|^r\\
&\le
C_r + C_r\sum_{i=0}^{k-1} h\,M_i,
\end{align*}
and
\begin{align*}
\mathbb E\max_{1\le m\le k}
\left|
\sum_{i=0}^{m-1} L_1b(t_i,X_i)\,I_i(W,W)
\right|^r
&\le
C_r\sum_{i=0}^{k-1} h\,\mathbb E|L_1b(t_i,X_i)|^r\\
&\le
C_r + C_r\sum_{i=0}^{k-1} h\,M_i.
\end{align*}
Consequently,
\[
M_k \le C_r + C_r\sum_{i=0}^{k-1} h\,M_i,
\qquad k=0,1,\dots,n.
\]
The discrete Gronwall lemma yields
\[
M_k\le C_r
\qquad\text{for all }k=0,1,\dots,n,
\]
with a constant independent of \(n\). This proves \eqref{eq:exact-scheme-max-moment-bound}.
Estimate \eqref{eq:exact-scheme-moment-bound} follows immediately.
\end{proof}

\begin{proposition}
\label{prop:stability-noisy}
Let \(r\ge 2\). Then there exists a constant \(C_r<\infty\), depending only on
\(r\), \(T\), \(K\), \(\gamma_1\), and \(\gamma_2\), such that for every
\((a,b,\eta)\in\mathcal F(\gamma_1,\gamma_2,K)\), every
\[
\tilde a\in V_a(\delta_1),\qquad
\tilde b\in V_b(\delta_2),\qquad
\widetilde{b_y}\in V_{b,b_y}(\delta_2),\qquad
\tilde W\in\mathcal W(W,\delta_3),
\]
and every \(n\in\mathbb N\),
\begin{equation}
\label{eq:stability-bound}
\mathbb E\max_{0\le i\le n}|X_i-\widetilde X_i|^r
\le
C_r\bigl(\delta_1^r+\delta_2^r+\delta_3^r\bigr).
\end{equation}
\end{proposition}

\begin{proof}
Define
\[
\mathcal E_k:=\mathbb E\max_{0\le j\le k}|e_j|^r,
\qquad k=0,1,\dots,n.
\]
By Proposition~\ref{prop:error-decomposition},
\[
e_k
=
\sum_{i=0}^{k-1}
\Big(
S_i^{(a)}+S_i^{(b)}+S_i^{(L)}
+
N_i^{(a)}+N_i^{(b)}+N_i^{(L)}
+
N_i^{(W,1)}+N_i^{(W,2)}+N_i^{(W,3)}
\Big).
\]
Hence, by the inequality \((x_1+\cdots+x_9)^r\le C_r\sum_{\ell=1}^9 |x_\ell|^r\),
\begin{equation}
\label{eq:E-split}
\mathcal E_k
\le
C_r\sum_{\star}
\mathbb E\max_{1\le m\le k}
\left|
\sum_{i=0}^{m-1} (\star)_i
\right|^r,
\end{equation}
where the sum extends over the nine families
\[
S^{(a)},\ S^{(b)},\ S^{(L)},\ N^{(a)},\ N^{(b)},\ N^{(L)},\ N^{(W,1)},\ N^{(W,2)},\ N^{(W,3)}.
\]

We estimate these terms one by one.

For \(S_i^{(a)}\), by Lipschitz continuity of \(a\),
\[
|S_i^{(a)}|\le K\,|e_i|\,h.
\]
Therefore, by Jensen's inequality,
\begin{align}
\mathbb E\max_{1\le m\le k}
\left|
\sum_{i=0}^{m-1} S_i^{(a)}
\right|^r
&\le
\mathbb E\Big(
\sum_{i=0}^{k-1} K|e_i|\,h
\Big)^r \notag\\
&\le
C_r\sum_{i=0}^{k-1} h\,\mathbb E|e_i|^r
\le
C_r\sum_{i=0}^{k-1} h\,\mathcal E_i. \label{eq:S-a-bound}
\end{align}

For \(S_i^{(b)}\), using the Lipschitz property of \(b\) and
Lemma~\ref{lem:weighted-increment-estimates},
\begin{align}
\mathbb E\max_{1\le m\le k}
\left|
\sum_{i=0}^{m-1} S_i^{(b)}
\right|^r
&\le
C_r\sum_{i=0}^{k-1} h\,\mathbb E|b(t_i,X_i)-b(t_i,\widetilde X_i)|^r \notag\\
&\le
C_r\sum_{i=0}^{k-1} h\,\mathbb E|e_i|^r
\le
C_r\sum_{i=0}^{k-1} h\,\mathcal E_i. \label{eq:S-b-bound}
\end{align}

Similarly, since \(L_1b\) is Lipschitz in the spatial variable by assumption,
Lemma~\ref{lem:weighted-increment-estimates} yields
\begin{align}
\mathbb E\max_{1\le m\le k}
\left|
\sum_{i=0}^{m-1} S_i^{(L)}
\right|^r
&\le
C_r\sum_{i=0}^{k-1} h\,\mathbb E|L_1b(t_i,X_i)-L_1b(t_i,\widetilde X_i)|^r \notag\\
&\le
C_r\sum_{i=0}^{k-1} h\,\mathbb E|e_i|^r
\le
C_r\sum_{i=0}^{k-1} h\,\mathcal E_i. \label{eq:S-L-bound}
\end{align}

Before bounding the specific noise terms, let us note a crucial algebraic simplification. 
Since the precision parameters satisfy $\delta_j \in [0,1]$ for $j \in \{1,2,3\}$, 
we have $\delta_j^r \le 1$. Whenever a term containing $\delta_j^r \sum h \mathcal E_i$ 
appears in our estimates, we bound it from above by simply absorbing $\delta_j^r$ 
into the generic constant $C_r$. This allows us to keep the implicit constant 
in the discrete Gronwall inequality independent of the noise levels.

Since \(|\widetilde X_i|\le |X_i|+|e_i|\), Proposition~\ref{prop:exact-scheme-moments} implies
\begin{equation}
\label{eq:Xtilde-growth-via-ex-e}
\mathbb E(1+|\widetilde X_i|)^r
\le
C_r\bigl(1+\mathbb E|X_i|^r+\mathbb E|e_i|^r\bigr)
\le
C_r\bigl(1+\mathcal E_i\bigr).
\end{equation}

For \(N_i^{(a)}\), estimate \eqref{eq:noise-a} and Jensen's inequality imply
\begin{align}
\mathbb E\max_{1\le m\le k}
\left|
\sum_{i=0}^{m-1} N_i^{(a)}
\right|^r
&\le
\mathbb E\Big(
\sum_{i=0}^{k-1} \delta_1(1+|\widetilde X_i|)\,h
\Big)^r \notag\\
&\le
C_r\,\delta_1^r \sum_{i=0}^{k-1} h\,\mathbb E(1+|\widetilde X_i|)^r \notag\\
&\le
C_r\,\delta_1^r
+
C_r\sum_{i=0}^{k-1} h\,\mathcal E_i. \label{eq:N-a-bound}
\end{align}

For \(N_i^{(b)}\), using \eqref{eq:noise-b},
\eqref{eq:Xtilde-growth-via-ex-e}, and \eqref{eq:weighted-dW},
\begin{align}
\mathbb E\max_{1\le m\le k}
\left|
\sum_{i=0}^{m-1} N_i^{(b)}
\right|^r
&\le
C_r\sum_{i=0}^{k-1} h\,\mathbb E|b(t_i,\widetilde X_i)-\tilde b(t_i,\widetilde X_i)|^r \notag\\
&\le
C_r\,\delta_2^r \sum_{i=0}^{k-1} h\,\mathbb E(1+|\widetilde X_i|)^r \notag\\
&\le
C_r\,\delta_2^r
+
C_r\sum_{i=0}^{k-1} h\,\mathcal E_i. \label{eq:N-b-bound}
\end{align}

For \(N_i^{(L)}\), using \eqref{eq:noise-L1b},
\eqref{eq:Xtilde-growth-via-ex-e}, and \eqref{eq:weighted-I},
\begin{align}
\mathbb E\max_{1\le m\le k}
\left|
\sum_{i=0}^{m-1} N_i^{(L)}
\right|^r
&\le
C_r\sum_{i=0}^{k-1} h\,\mathbb E|L_1b(t_i,\widetilde X_i)-\widetilde{L_1b}(t_i,\widetilde X_i)|^r \notag\\
&\le
C_r\,\delta_2^r \sum_{i=0}^{k-1} h\,\mathbb E(1+|\widetilde X_i|)^r \notag\\
&\le
C_r\,\delta_2^r
+
C_r\sum_{i=0}^{k-1} h\,\mathcal E_i. \label{eq:N-L-bound}
\end{align}

Using \eqref{eq:growth-noisy}, \eqref{eq:Xtilde-growth-via-ex-e}, and \eqref{eq:weighted-dZ},
we obtain
\begin{align}
\mathbb E\max_{1\le m\le k}
\left|
\sum_{i=0}^{m-1} N_i^{(W,1)}
\right|^r
&\le
C_r\,\delta_3^r \sum_{i=0}^{k-1} h\,\mathbb E|\tilde b(t_i,\widetilde X_i)|^r \notag\\
&\le
C_r\,\delta_3^r
+
C_r\sum_{i=0}^{k-1} h\,\mathcal E_i. \label{eq:W1-bound}
\end{align}

Similarly, using \eqref{eq:weighted-dWdZ},
\begin{align}
\mathbb E\max_{1\le m\le k}
\left|
\sum_{i=0}^{m-1} N_i^{(W,2)}
\right|^r
&\le
C_r\,\delta_3^r \sum_{i=0}^{k-1} h\,\mathbb E|\widetilde{L_1b}(t_i,\widetilde X_i)|^r \notag\\
&\le
C_r\,\delta_3^r
+
C_r\sum_{i=0}^{k-1} h\,\mathcal E_i. \label{eq:W2-bound}
\end{align}

Finally, by \eqref{eq:weighted-dZ2},
\begin{align}
\mathbb E\max_{1\le m\le k}
\left|
\sum_{i=0}^{m-1} N_i^{(W,3)}
\right|^r
&\le
C_r\,\delta_3^{2r} \sum_{i=0}^{k-1} h\,\mathbb E|\widetilde{L_1b}(t_i,\widetilde X_i)|^r \notag\\
&\le
C_r\,\delta_3^{2r}
+
C_r\sum_{i=0}^{k-1} h\,\mathcal E_i \notag\\
&\le
C_r\,\delta_3^r
+
C_r\sum_{i=0}^{k-1} h\,\mathcal E_i, \label{eq:W3-bound}
\end{align}
since \(\delta_3\in[0,1]\).

\medskip
Combining \eqref{eq:E-split} with \eqref{eq:S-a-bound}--\eqref{eq:W3-bound}, we arrive at
\[
\mathcal E_k
\le
C_r\sum_{i=0}^{k-1} h\,\mathcal E_i
+
C_r\bigl(\delta_1^r+\delta_2^r+\delta_3^r\bigr),
\qquad k=0,1,\dots,n.
\]
The discrete Gronwall lemma now yields
\[
\mathcal E_k
\le
C_r\bigl(\delta_1^r+\delta_2^r+\delta_3^r\bigr),
\qquad k=0,1,\dots,n,
\]
with a constant independent of \(n\). In particular, \eqref{eq:stability-bound} follows for \(k=n\).
\end{proof}

As an immediate consequence, we obtain the perturbation part of the final upper bound.

\begin{corollary}
\label{cor:noise-upper}
Let \(r\ge 2\). Under the assumptions of Proposition~\ref{prop:stability-noisy},
\begin{equation}
\label{eq:noise-upper}
\left\|
X_n^{RM}-\widetilde X_n^{RM}
\right\|_r
\le
C_r\,(\delta_1+\delta_2+\delta_3),
\end{equation}
where
\[
X_n^{RM}:=X^{RM}_{n,n},
\qquad
\widetilde X_n^{RM}:=\widetilde X^{RM}_{n,n}.
\]
\end{corollary}

\begin{proof}
By Proposition~\ref{prop:stability-noisy},
\[
\left\|X_n^{RM}-\widetilde X_n^{RM}\right\|_r^r
\le
\mathbb E\max_{0\le i\le n}|X_i-\widetilde X_i|^r
\le
C_r\bigl(\delta_1^r+\delta_2^r+\delta_3^r\bigr).
\]
Taking the \(r\)-th root and using \(\delta_j\in[0,1]\) gives \eqref{eq:noise-upper}.
\end{proof}

Let
\[
\widehat K:=\max\{K,K^{2r}\}.
\]
The assumptions defining \(\mathcal F(\gamma_1,\gamma_2,K)\) are at least as strong as the assumptions required in \cite[Proposition~1]{PMPP19}, applied with the moment parameter \(q=r\) and the class constant \(\widehat K\). Indeed, the coefficient conditions on \(a\), \(b\), and \(L_1b\) match the required spatial Lipschitz and temporal H\"older assumptions, and
\[
\mathbb E|\eta|^{2r}\le K^{2r}\le\widehat K.
\]
Moreover, the auxiliary exact-information scheme \eqref{eq:exact-rm} coincides at the grid points with the randomized Milstein scheme covered by \cite[Proposition~1]{PMPP19}. Consequently,
\begin{equation}
\label{eq:exact-info-rate}
\sup_{(a,b,\eta)\in\mathcal F(\gamma_1,\gamma_2,K)}
\left\|
X^{a,b,\eta}(T)-X_n^{RM}
\right\|_r
\le
C_r\,n^{-\min\{\gamma_1+\frac12,\gamma_2\}}.
\end{equation}
Combining \eqref{eq:noise-upper} and \eqref{eq:exact-info-rate}, the noisy randomized Milstein scheme satisfies the global upper bound
\begin{multline}
\label{eq:final-upper-alg}
\sup_{(a,b,\eta)\in\mathcal F(\gamma_1,\gamma_2,K)}
\ \sup_{\tilde a\in V_a(\delta_1)}
\ \sup_{\tilde b\in V_b(\delta_2)}
\ \sup_{\widetilde{b_y}\in V_{b,b_y}(\delta_2)}
\ \sup_{\tilde W\in\mathcal W(W,\delta_3)}
\left\|
X^{a,b,\eta}(T)-\widetilde X_n^{RM}
\right\|_r
\le \\
C_r\Bigl(
n^{-\min\{\gamma_1+\frac12,\gamma_2\}}
+\delta_1+\delta_2+\delta_3
\Bigr).
\end{multline}
Since \(\mathcal A_n^{RM,\mathrm{noisy}}\in\Phi_{4n+1}\), this implies
\begin{equation}
\label{eq:final-upper-minimal}
e_{4n+1}^{(r)}\bigl(\mathcal F(\gamma_1,\gamma_2,K),\delta_1,\delta_2,\delta_3\bigr)
\le
C_r\Bigl(
n^{-\min\{\gamma_1+\frac12,\gamma_2\}}
+\delta_1+\delta_2+\delta_3
\Bigr).
\end{equation}

\section{Lower bounds and optimality}\label{sec:lower}

Set
\begin{equation}
\label{eq:alpha-def-lower}
\alpha:=\min\left\{\gamma_1+\frac12,\gamma_2\right\}.
\end{equation}
We first record the part of the lower bound that is already contained in the standard noisy-coefficient theory. The result below is a direct consequence of \cite[Lemma~3]{PMPP19}. Its proof in that paper uses the usual deterministic-integration and It\^o-integration subproblems, together with constant perturbations of the drift and diffusion information. These subproblems are contained in our present class; the additional oracle for \(b_y\) is identically zero on them and hence carries no additional information.

\begin{proposition}
\label{prop:lower-PMPP19}
There exist constants \(c_{12}>0\) and \(N_{12}\in\mathbb N\), depending only on \(K,T,\gamma_1,\gamma_2\), such that for all \(N\ge N_{12}\) and all \(\delta_1,\delta_2,\delta_3\in[0,1]\),
\begin{equation}
\label{eq:lower-PMPP19}
e_N^{(2)}\bigl(\mathcal F(\gamma_1,\gamma_2,K),\delta_1,\delta_2,\delta_3\bigr)
\ge
c_{12}\,\max\bigl\{N^{-\alpha},\delta_1,\delta_2\bigr\}.
\end{equation}
\end{proposition}

\begin{proof}
Since \(p_W\equiv0\) is admissible in \(\mathcal W(W,\delta_3)\), and since exact derivative information \(\widetilde{b_y}=b_y\) is admissible in \(V_{b,b_y}(\delta_2)\), the present worst-case problem contains the noisy-coefficient problem considered in \cite[Lemma~3]{PMPP19} as a subproblem. More precisely, the lower bounds generated there by the drift quadrature subproblem, the stochastic It\^o integration subproblem, and the constant perturbations of \(a\) and \(b\) apply verbatim to our class of standard-information algorithms. Therefore, for \(q=2\),
\[
e_N^{(2)}\bigl(\mathcal F(\gamma_1,\gamma_2,K),\delta_1,\delta_2,\delta_3\bigr)
\ge
c_{12}\,\max\bigl\{N^{-\alpha},\delta_1,\delta_2\bigr\}
\]
for all sufficiently large \(N\), after adjusting the constant to the present parameters of the class.
\end{proof}

The lower bound proportional to \(\delta_2\) should be interpreted as a lower bound for the
diffusion-information channel. It is already obtained by perturbing the diffusion coefficient
\(b\) itself. We do not claim that perturbations of the derivative oracle \(b_y\) alone necessarily
produce a minimax error of order \(\delta_2\).

We now prove the only lower-bound contribution that is not covered by the noisy-coefficient model, namely the contribution of the perturbation of the observed Wiener path.

\begin{lemma}
\label{lem:LeCam-constant-gap}
Let \(Q_0,Q_1\) be probability measures on a measurable space \((\mathcal X,\mathcal B)\). Let \(\vartheta_0,\vartheta_1:\mathcal X\to\mathbb R\) be measurable and suppose that
\[
\vartheta_0(x)-\vartheta_1(x)\equiv \Delta,
\qquad x\in\mathcal X,
\]
for some \(\Delta>0\). Then, for every measurable estimator \(Y:\mathcal X\to\mathbb R\),
\[
\max\bigl\{\|Y-\vartheta_0\|_{2,Q_0},\ \|Y-\vartheta_1\|_{2,Q_1}\bigr\}
\ge
\frac{\Delta}{2}\bigl(1-\TV(Q_0,Q_1)\bigr),
\]
where
\[
\TV(Q_0,Q_1):=\sup_{A\in\mathcal B}|Q_0(A)-Q_1(A)|.
\]
\end{lemma}

\begin{proof}
We first prove the corresponding $L^1$ bound and then use
$\|Z\|_{2,Q}\ge \E_Q|Z|$.
Let $\mu:=Q_0+Q_1$ and let $q_i:=\rd Q_i/\rd\mu$.
Define the overlap measure $M$ by
\[
\rd M:=\min\{q_0,q_1\}\,\rd\mu.
\]
Then
\[
M(\mathcal X)=\int \min\{q_0,q_1\}\,\rd\mu.
\]
Moreover, since
\[
q_0+q_1 = |q_0-q_1| + 2\min\{q_0,q_1\},
\]
integration with respect to \(\mu\) yields
\[
2
=
\int (q_0+q_1)\,\rd\mu
=
\int |q_0-q_1|\,\rd\mu
+
2\int \min\{q_0,q_1\}\,\rd\mu.
\]
Using the standard density representation of the total variation distance,
\[
\TV(Q_0,Q_1)=\frac12\int |q_0-q_1|\,\rd\mu,
\]
we obtain
\[
M(\mathcal X)
=
\int \min\{q_0,q_1\}\,\rd\mu
=
1-\TV(Q_0,Q_1).
\]
For every $x\in\mathcal X$, by the triangle inequality,
\[
|Y(x)-\vartheta_0(x)|+|Y(x)-\vartheta_1(x)|
\ \ge\
|\vartheta_0(x)-\vartheta_1(x)|
=
\Delta.
\]
Integrating with respect to $M$ yields
\[
\int |Y-\vartheta_0|\,\rd M
+
\int |Y-\vartheta_1|\,\rd M
\ \ge\
\Delta\,M(\mathcal X)
=
\Delta\,(1-\TV(Q_0,Q_1)).
\]
Since $M\le Q_0$ and $M\le Q_1$ as measures, we have
\[
\int |Y-\vartheta_0|\,\rd M\le \E_{Q_0}|Y-\vartheta_0|,
\qquad
\int |Y-\vartheta_1|\,\rd M\le \E_{Q_1}|Y-\vartheta_1|.
\]
Hence
\[
\E_{Q_0}|Y-\vartheta_0|+\E_{Q_1}|Y-\vartheta_1|
\ \ge\
\Delta\,(1-\TV(Q_0,Q_1)).
\]
Finally, using $\max\{u,v\}\ge (u+v)/2$, we get
\[
\max\big\{\E_{Q_0}|Y-\vartheta_0|,\ \E_{Q_1}|Y-\vartheta_1|\big\}
\ \ge\
\frac{\Delta}{2}\,(1-\TV(Q_0,Q_1)),
\]
and therefore also
\[
\max\big\{\|Y-\vartheta_0\|_{2,Q_0},\ \|Y-\vartheta_1\|_{2,Q_1}\big\}
\ \ge\
\frac{\Delta}{2}\,(1-\TV(Q_0,Q_1)).
\qedhere
\]
\end{proof}

The next lower bound is not a discretization argument. We will give the estimator strictly
more information than is available in the standard-information model, namely the entire
corrupted path \(\widetilde W\). Even under this stronger information, the terminal value
cannot be recovered with accuracy better than order \(\delta_3\), because the observation laws
corresponding to two admissible Wiener-path perturbations are statistically close.

\begin{proposition}
\label{prop:lower-delta3}
There exists a constant \(c_3=c_3(K,T)>0\) such that, for every \(N\in\mathbb N\) and all \(\delta_1,\delta_2,\delta_3\in[0,1]\),
\begin{equation}
\label{eq:lower-delta3}
e_N^{(2)}\bigl(\mathcal F(\gamma_1,\gamma_2,K),\delta_1,\delta_2,\delta_3\bigr)
\ge
c_3\delta_3.
\end{equation}
One may take
\[
c_3=\frac14\min\{K,1\}\min\{T,1\}.
\]
\end{proposition}

\begin{proof}
Let \(\beta:=\min\{K,1\}\) and restrict the input class to
\[
a\equiv0,
\qquad
b\equiv\beta,
\qquad
\eta=0.
\]
This input belongs to \(\mathcal F(\gamma_1,\gamma_2,K)\), and the corresponding solution satisfies
\[
X(T)=\beta W(T).
\]
We take exact coefficient information,
\[
\tilde a=a,
\qquad
\tilde b=b,
\qquad
\widetilde{b_y}=b_y,
\]
which is admissible in \(V_a(\delta_1)\), \(V_b(\delta_2)\), and \(V_{b,b_y}(\delta_2)\). Thus only the perturbed Wiener information remains relevant.

Fix an arbitrary \(\mathcal A\in\Phi_N\). Let \(\zeta\) denote all auxiliary randomness used by \(\mathcal A\). By the construction of randomized standard information, \(\zeta\) is independent of \(\Sigma_\infty\). Enlarging the information can only decrease the minimal error, so we may allow the estimator to observe the full corrupted path \(\tilde W\) and \(\zeta\). Hence it is enough to consider arbitrary measurable estimators
\[
Y:C([0,T])\times\mathcal R\to\mathbb R,
\]
where \((\mathcal R,\mathcal B_{\mathcal R},Q_\zeta)\) is the distribution space of \(\zeta\). Consequently,
\begin{equation}
\label{eq:delta3-full-path-reduction}
e_N^{(2)}\bigl(\mathcal F(\gamma_1,\gamma_2,K),\delta_1,\delta_2,\delta_3\bigr)
\ge
\inf_Y\sup_{\tilde W\in\mathcal W(W,\delta_3)}
\|\beta W(T)-Y(\tilde W,\zeta)\|_2.
\end{equation}

Consider
\[
g(t):=\frac{t}{\max\{T,1\}},
\qquad
u(t):=\delta_3g(t),
\qquad t\in[0,T].
\]
The functions \(p_0(t,x)\equiv0\) and \(p_1(t,x):=g(t)\) belong to \(\widetilde{\mathcal K}\), since \(p_i(0,0)=0\), \(p_{i,x}=p_{i,xx}\equiv0\), and \(|p_{i,t}|\le1\). Therefore the two observations
\[
\tilde W_0(t)=W(t),
\qquad
\tilde W_1(t)=W(t)+u(t),
\qquad t\in[0,T],
\]
are admissible elements of \(\mathcal W(W,\delta_3)\). Let \(P_i\) be the law of \(\tilde W_i\) on \(C([0,T])\), and put
\[
\bar P_i:=P_i\otimes Q_\zeta,
\qquad i=0,1.
\]
The product form follows from the independence of \(\zeta\) and \(\Sigma_\infty\). On \(\mathcal X:=C([0,T])\times\mathcal R\), define
\[
\vartheta_0(\omega,z):=\beta\omega(T),
\qquad
\vartheta_1(\omega,z):=\beta\bigl(\omega(T)-u(T)\bigr).
\]
Under \(\bar P_i\), the quantity \(\vartheta_i\) is exactly the target \(\beta W(T)\) expressed as a function of the observed corrupted path. Moreover,
\begin{equation}
\label{eq:delta3-gap}
\vartheta_0-\vartheta_1\equiv\Delta,
\qquad
\Delta:=\beta u(T)=\beta\delta_3\min\{T,1\}.
\end{equation}
From \eqref{eq:delta3-full-path-reduction}, restricting the supremum to \(\tilde W_0\) and \(\tilde W_1\), and applying Lemma~\ref{lem:LeCam-constant-gap}, we obtain
\begin{equation}
\label{eq:delta3-lecam-applied}
e_N^{(2)}\bigl(\mathcal F(\gamma_1,\gamma_2,K),\delta_1,\delta_2,\delta_3\bigr)
\ge
\frac{\Delta}{2}\bigl(1-\TV(\bar P_0,\bar P_1)\bigr).
\end{equation}
Since \(\bar P_i=P_i\otimes Q_\zeta\),
\[
\TV(\bar P_0,\bar P_1)=\TV(P_0,P_1).
\]
The shift \(u\) belongs to the Cameron--Martin space \(H^1_0([0,T])\). If \(B_t(\omega)=\omega(t)\) denotes the canonical process on \(C([0,T])\), then, by the Cameron--Martin theorem,
\begin{equation}
\label{eq:delta3-RN}
\frac{\rd P_1}{\rd P_0}
=
\exp\left(
\int_0^T \dot u(t)\,\rd B_t
-
\frac12\int_0^T |\dot u(t)|^2\,\rd t
\right)
\qquad P_0\text{-a.s.}
\end{equation}
see, for instance, \cite[Ch.~5]{CohEl} or \cite[Ch.~VIII]{Protter}. Hence
\[
D_{\rm KL}(P_1\|P_0)
=
\frac12\int_0^T |\dot u(t)|^2\,\rd t
=
\frac{\delta_3^2}{2}\int_0^T |g'(t)|^2\,\rd t
\le
\frac{\delta_3^2}{2},
\]
because \(g'(t)=1/\max\{T,1\}\). Pinsker's inequality gives
\[
\TV(\bar P_0,\bar P_1)=\TV(P_0,P_1)
\le
\sqrt{\frac12D_{\rm KL}(P_1\|P_0)}
\le
\frac{\delta_3}{2}.
\]
Since \(\delta_3\in[0,1]\), \(1-\TV(\bar P_0,\bar P_1)\ge1/2\). Combining this with \eqref{eq:delta3-lecam-applied} and \eqref{eq:delta3-gap} yields
\[
e_N^{(2)}\bigl(\mathcal F(\gamma_1,\gamma_2,K),\delta_1,\delta_2,\delta_3\bigr)
\ge
\frac14\min\{K,1\}\min\{T,1\}\delta_3.
\]
\end{proof}

\begin{remark}
\label{rem:method-independent-delta3}
The proof of Proposition~\ref{prop:lower-delta3} uses neither the Milstein correction nor any
property of the time discretization. It only uses the observation model and the scalar test
equation
\[
\rd X(t)=\beta\,\rd W(t).
\]
Moreover, the proof remains valid after granting the estimator the entire corrupted path.
Consequently, the same lower bound applies to any method based on the corresponding noisy
Wiener information.

In particular, consider the scalar case of the randomized Euler standard-information model
from \cite{BaranekEtAl2026} with its smooth perturbation class \(\mathcal W_0\). Combining
Proposition~\ref{prop:lower-delta3} with the upper bound and the coefficient-noise lower bounds
in \cite[Theorem~2]{BaranekEtAl2026} yields
\[
e_N^{(r)}
\bigl(
\mathcal F(\varrho,K),
\mathcal W_0,
\delta_1,\delta_2,\delta_3
\bigr)
\asymp
N^{-\min\{\varrho,1/2\}}
+\delta_1+\delta_2+\delta_3,
\qquad r\ge2.
\]
Thus, in the scalar smooth-noise setting, the two-point construction fills the
\(\delta_3\)-gap left open in the general bounds of \cite{BaranekEtAl2026}.
\end{remark}

\begin{theorem}
\label{thm:global-lower}
Let \(r\ge2\). There exists a constant \(c>0\), depending only on \(r,K,T,\gamma_1,\gamma_2\), such that, for every \(N\in\mathbb N\) and all \(\delta_1,\delta_2,\delta_3\in[0,1]\),
\begin{equation}
\label{eq:global-lower}
e_N^{(r)}\bigl(\mathcal F(\gamma_1,\gamma_2,K),\delta_1,\delta_2,\delta_3\bigr)
\ge
c\bigl(N^{-\alpha}+\delta_1+\delta_2+\delta_3\bigr).
\end{equation}
\end{theorem}

\begin{proof}
For \(r\ge2\), monotonicity of \(L^p\)-norms gives \(e_N^{(r)}\ge e_N^{(2)}\). If \(N\ge N_{12}\), Proposition~\ref{prop:lower-PMPP19} and Proposition~\ref{prop:lower-delta3} imply
\[
e_N^{(2)}
\ge
c_0\max\bigl\{N^{-\alpha},\delta_1,\delta_2,\delta_3\bigr\}
\]
with \(c_0:=\min\{c_{12},c_3\}\). For \(N<N_{12}\), monotonicity of the minimal error with respect to the information budget gives \(e_N^{(2)}\ge e_{N_{12}}^{(2)}\), while Proposition~\ref{prop:lower-delta3} still gives the \(\delta_3\)-term. Hence, after decreasing \(c_0\) by the factor \(N_{12}^{-\alpha}\), the same estimate holds for all \(N\in\mathbb N\). Finally,
\[
\max\{x_1,x_2,x_3,x_4\}
\ge
\frac14(x_1+x_2+x_3+x_4),
\qquad x_j\ge0,
\]
which proves \eqref{eq:global-lower}.
\end{proof}

\begin{corollary}
\label{cor:minimax-optimality}
Let \(r\ge2\). There exist constants \(0<c\le C<\infty\), depending only on \(r,K,T,\gamma_1,\gamma_2\), such that, for every \(N\in\mathbb N\) and all \(\delta_1,\delta_2,\delta_3\in[0,1]\),
\begin{multline}
\label{eq:minimax-optimality}
c\bigl(N^{-\alpha}+\delta_1+\delta_2+\delta_3\bigr)
\le
e_N^{(r)}\bigl(\mathcal F(\gamma_1,\gamma_2,K),\delta_1,\delta_2,\delta_3\bigr)
\le{}\\
C\bigl(N^{-\alpha}+\delta_1+\delta_2+\delta_3\bigr).
\end{multline}
Consequently, the noisy randomized Milstein scheme is minimax order-optimal in the
randomized standard-information model defined in Section~\ref{subsec:adaptive-information}.
\end{corollary}

\begin{proof}
The lower bound is Theorem~\ref{thm:global-lower}. For the upper bound, let \(N\ge9\) and set
\[
n_N:=\left\lfloor\frac{N-1}{4}\right\rfloor.
\]
Then \(4n_N+1\le N\) and \(n_N\ge N/9\). By monotonicity of the minimal error and by \eqref{eq:final-upper-minimal},
\[
e_N^{(r)}
\le
e_{4n_N+1}^{(r)}
\le
C_r\bigl(n_N^{-\alpha}+\delta_1+\delta_2+\delta_3\bigr)
\le
C\bigl(N^{-\alpha}+\delta_1+\delta_2+\delta_3\bigr).
\]
For the finitely many cases \(1\le N<9\), we use the zero algorithm. Standard moment estimates for solutions of SDEs with coefficients from \(\mathcal F(\gamma_1,\gamma_2,K)\) give
\[
\sup_{(a,b,\eta)\in\mathcal F(\gamma_1,\gamma_2,K)}
\|X^{a,b,\eta}(T)\|_r<\infty,
\]
and since \(N^{-\alpha}\ge 8^{-\alpha}\) for \(N<9\), these finitely many budgets are absorbed into the constant \(C\).
\end{proof}
\section{Numerical experiments}\label{sec:numerical}

The goal of this section is to illustrate the four-term error structure predicted by the theoretical results,
\[
 n^{-\alpha},\qquad \delta_1,\qquad \delta_2,\qquad \delta_3,
 \qquad
 \alpha=\min\{\gamma_1+\tfrac12,\gamma_2\}.
\]

Since the information cost of the noisy randomized Milstein scheme is \(4n+1\), rates in the
number of time steps \(n\) and in the information budget \(N\) are equivalent up to constant
factors.

All reported errors are root mean square errors. All Monte Carlo estimates were computed
from \(10000\) independent sample paths; the shaded bands in the figures represent two
estimated standard errors of the Monte Carlo mean. In the noisy experiments, active
perturbation channels are multiplied by signs from a finite set and the largest observed
RMSE over the tested sign configurations is reported. This finite maximization is not part
of the theoretical definition; it is used only as a numerical proxy for the worst-case character
of the error criterion.

\subsection{Benchmark equation and admissible perturbations}

As the main test problem we use the exactly solvable linear SDE
\begin{equation}
\label{eq:numerical-linear-benchmark}
\begin{cases}
 \rd X(t)=\mu(t)X(t)\,\rd t+\sigma(t)X(t)\,\rd W(t),\qquad t\in[0,T],\\[0.2em]
 X(0)=x_0,
\end{cases}
\end{equation}
where
\[
T=1,
\qquad
x_0=1,
\qquad
\mu(t)=\mu_0+\mu_1t^{\gamma_1},
\qquad
\sigma(t)=\sigma_0+\sigma_1t^{\gamma_2},
\]
and
\[
\mu_0=0.7,
\qquad
\mu_1=0.35,
\qquad
\sigma_0=0.4,
\qquad
\sigma_1=0.3,
\qquad
\gamma_1=0.2,
\qquad
\gamma_2=0.7.
\]
Thus
\[
\alpha=\min\{\gamma_1+\tfrac12,\gamma_2\}=0.7.
\]
The coefficients are
\[
a(t,x)=\mu(t)x,
\qquad
b(t,x)=\sigma(t)x,
\qquad
b_y(t,x)=\sigma(t),
\qquad
L_1b(t,x)=\sigma(t)^2x,
\]
and hence satisfy the assumptions of the input class for a sufficiently large value of the parameter \(K\).  The advantage of this benchmark is that the exact terminal solution is available:
\begin{equation}
\label{eq:numerical-linear-exact}
X(T)
=
x_0\exp\left(
\int_0^T\left(\mu(t)-\frac12\sigma(t)^2\right)\,\rd t
+
\int_0^T\sigma(t)\,\rd W(t)
\right).
\end{equation}

For the exact terminal value, the stochastic integral
\[
\int_0^T \sigma(t)\,\rd W(t)
\]
was sampled jointly with the Brownian increments used by the numerical scheme. On each
interval \([t_i,t_{i+1}]\), the pair
\[
\left(\Delta W_i,\int_{t_i}^{t_{i+1}}\sigma(t)\,\rd W(t)\right)
\]
is Gaussian with covariance matrix determined by
\[
\operatorname{Var}(\Delta W_i)=h,
\qquad
\operatorname{Cov}\left(\Delta W_i,\int_{t_i}^{t_{i+1}}\sigma(t)\,\rd W(t)\right)
=
\int_{t_i}^{t_{i+1}}\sigma(t)\,\rd t,
\]
and
\[
\operatorname{Var}\left(\int_{t_i}^{t_{i+1}}\sigma(t)\,\rd W(t)\right)
=
\int_{t_i}^{t_{i+1}}\sigma(t)^2\,\rd t.
\]

 Consequently, the strong error is measured against the exact terminal value, not against a refined-grid approximation.

For coefficient noise we use the perturbation functions
\[
p_a(t,x)\equiv 1,
\qquad
p_b(t,x)\equiv 1,
\qquad
q_b(t,x)\equiv 1.
\]
Clearly, \(p_a,p_b\in\mathcal K_{\rm lin}\) and \(q_b\in\mathcal K_{\rm bd}\).  For the Wiener-path perturbation in the stability and total-error experiments we use
\begin{equation}
\label{eq:numerical-mixed-sin-perturbation}
p_W(t,x)=\frac12\sin(t+x),
\end{equation}
which belongs to \(\widetilde{\mathcal K}\), since \(|p_W(0,0)|=0\) and its first time derivative and first two spatial derivatives are bounded by \(1/2\) in absolute value.

\subsection{Linear response to noisy information}

We first isolate the perturbation part of the error.  For fixed \(n=1024\), the exact-information and noisy-information Milstein schemes are run on the same Brownian paths and the same randomized quadrature points.  We then estimate
\[
\left\|X_n^{RM}-\widetilde X_n^{RM}\right\|_2
\]
as a function of the active noise level \(\delta\).  Figure~\ref{fig:perturbation-response} shows the response for the drift oracle, the diffusion oracle, the Wiener path, and all channels acting simultaneously.

\begin{figure}[htbp]
\centering
\includegraphics[width=0.98\linewidth]{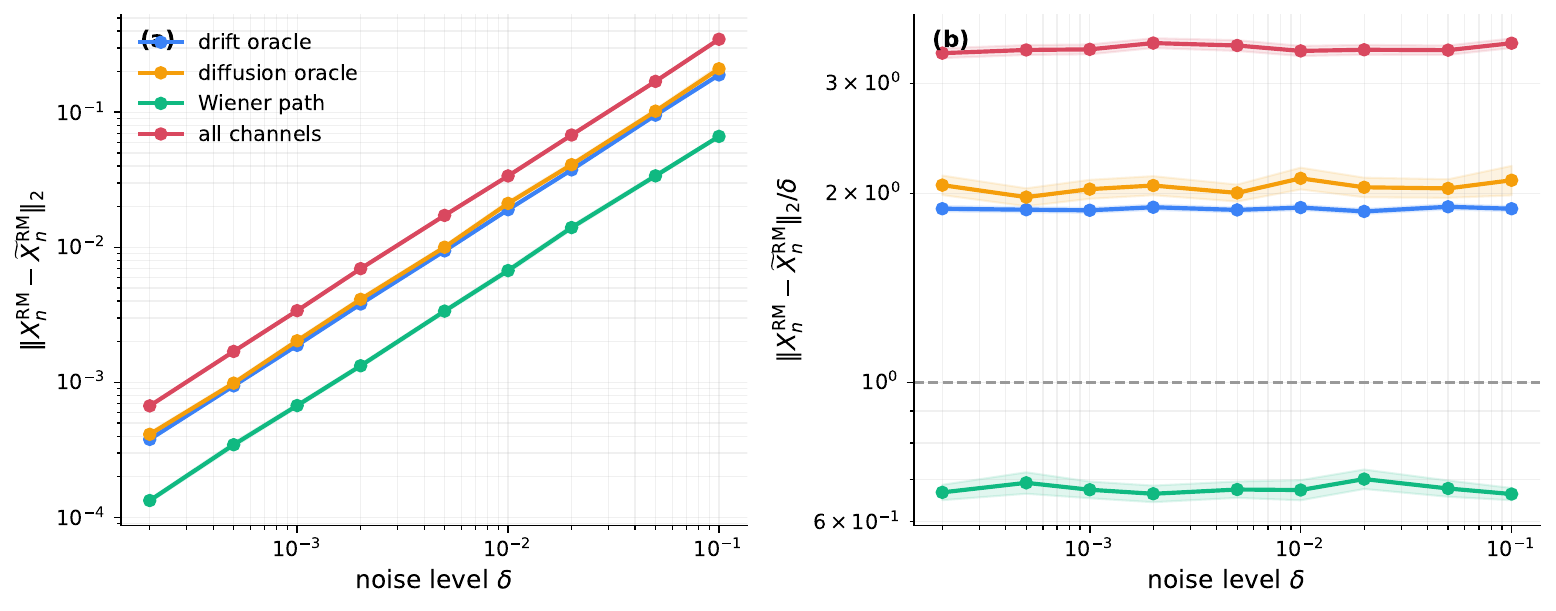}
\caption{Perturbation response of the noisy randomized Milstein scheme at a fixed discretization level. Left: RMSE between the exact-information and noisy-information schemes versus the common noise level \(\delta\). Right: the normalized quantity \(\|X_n^{RM}-\widetilde X_n^{RM}\|_2/\delta\), which remains approximately constant and illustrates the linear dependence on the perturbation size.}
\label{fig:perturbation-response}
\end{figure}
\FloatBarrier

The log-log slopes in \(\delta\) are essentially equal to one; see Table~\ref{tab:perturbation-slopes}.  This is the numerical counterpart of the stability estimate
\[
\left\|X_n^{RM}-\widetilde X_n^{RM}\right\|_2
\lesssim
\delta_1+\delta_2+\delta_3.
\]
The derivative-oracle perturbation was also tested separately.  Its empirical slope is again close to one, but its constant is much smaller, which is natural because the perturbation of \(b_y\) enters only through the Milstein correction.

\begin{table}[htbp]
\centering
\begin{tabular}{lc}
\hline
active perturbation channel & fitted slope in \(\delta\) \\
\hline
drift oracle & 1.000 \\
diffusion oracle & 1.004 \\
Wiener path & 1.000 \\
all channels & 1.002 \\
\hline
\end{tabular}
\caption{Empirical slopes for the perturbation-response experiment in Figure~\ref{fig:perturbation-response}.}
\label{tab:perturbation-slopes}
\end{table}

\subsection{Fixed noise levels and error floors}

We next study the total strong error
\[
\left\|X(T)-\widetilde X_n^{RM}\right\|_2
\]
for fixed nonzero noise levels and
\[
n\in\{32,64,128,256,512,1024\}.
\]
Figure~\ref{fig:noise-floor} shows the exact-information error and the noisy errors for fixed levels \(\delta=0.05\) and \(\delta=0.1\).  The noisy curves display the expected transition from a discretization-dominated regime to a noise-dominated regime: once the discretization error becomes smaller than the active perturbation contribution, the total error reaches a floor.

\begin{figure}[htbp]
\centering
\includegraphics[width=0.86\linewidth]{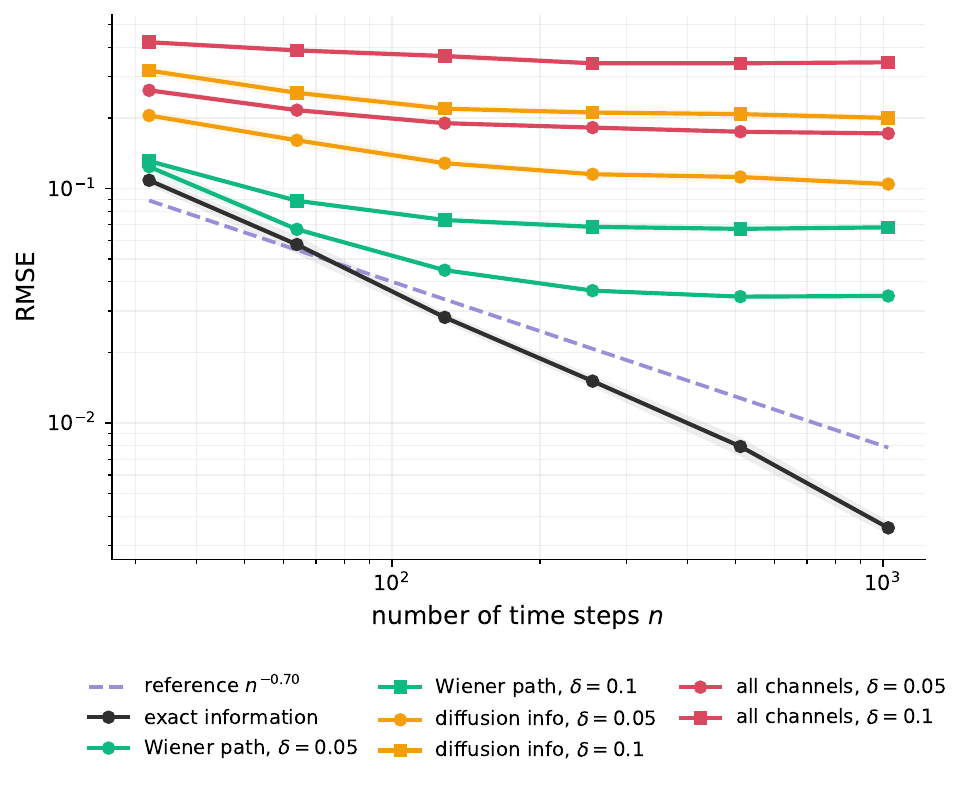}
\caption{Total RMSE for the exactly solvable linear benchmark. The exact-information scheme decays rapidly, whereas fixed perturbation levels generate the expected noise floors. The cases ``diffusion info'' and ``all channels'' use the same noise level in the noisy oracles for \(b\) and \(b_y\), while ``Wiener path'' refers to perturbations of the observed Brownian path. The dashed line has slope \(-\alpha\), where \(\alpha=0.7\).}
\label{fig:noise-floor}
\end{figure}
\FloatBarrier

The exact-information curve decreases faster than \(n^{-\alpha}\) on this particular benchmark, with fitted rate approximately \(0.97\).  This is not in contradiction with the theory: the rate \(n^{-\alpha}\) is a worst-case, minimax rate over the whole input class, whereas \eqref{eq:numerical-linear-benchmark} is one concrete, exactly solvable test equation.  The relevant effect in Figure~\ref{fig:noise-floor} is the formation of error floors under fixed noise.

\subsection{Balanced noise regime}

The optimality result suggests a natural balanced scaling.  We set
\begin{equation}
\label{eq:balanced-delta-def}
\delta_1(n)=\delta_2(n)=\delta_3(n)=c n^{-\alpha}.
\end{equation}
Then
\[
n^{-\alpha}+\delta_1(n)+\delta_2(n)+\delta_3(n)
=
(1+3c)n^{-\alpha},
\]
so the total error should follow the same order \(n^{-\alpha}\), with a larger constant depending on \(c\).  This behaviour is shown in Figure~\ref{fig:balanced-regime}.

\begin{figure}[htbp]
\centering
\includegraphics[width=0.86\linewidth]{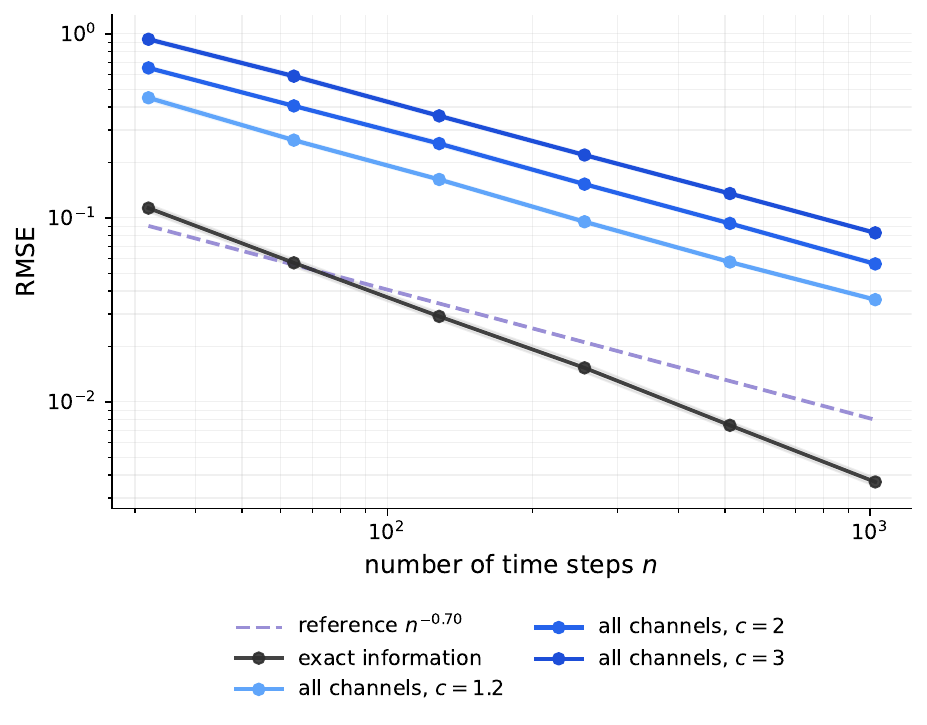}
\caption{Balanced regime for the exactly solvable linear benchmark. Here \(\delta_1(n)=\delta_2(n)=\delta_3(n)=c n^{-\alpha}\) with \(\alpha=0.7\). Although the exact-information scheme converges faster on this particular benchmark, the noisy schemes inherit the predicted order \(n^{-\alpha}\) once the perturbation levels are balanced with the discretization scale.}
\label{fig:balanced-regime}
\end{figure}
\FloatBarrier

The fitted rates are reported in Table~\ref{tab:balanced-rates}.  As \(c\) increases, the perturbation contribution becomes dominant and the fitted rate approaches the predicted value \(\alpha=0.7\).  For \(c=2\) and \(c=3\), the ratios
\[
\frac{\mathrm{RMSE}}{n^{-\alpha}+\delta_1(n)+\delta_2(n)+\delta_3(n)}
\]
are nearly constant over the tested range of \(n\); their observed ranges are approximately \([1.03,1.08]\) and \([1.06,1.08]\), respectively.

\begin{table}[htbp]
\centering
\begin{tabular}{lc}
\hline
scenario & fitted rate \\
\hline
exact information & 0.984 \\
balanced noise, \(c=1.2\) & 0.731 \\
balanced noise, \(c=2\) & 0.708 \\
balanced noise, \(c=3\) & 0.701 \\
\hline
\end{tabular}
\caption{Empirical rates in the balanced regime from Figure~\ref{fig:balanced-regime}.  The predicted minimax rate for the chosen parameters is \(\alpha=0.7\).}
\label{tab:balanced-rates}
\end{table}

\subsection{A nonlinear exactly solvable benchmark}

To check that the numerical behaviour above is not an artifact of the linear diffusion coefficient in \eqref{eq:numerical-linear-benchmark}, we also ran the same type of experiments on a nonlinear exactly solvable benchmark.  We consider
\begin{equation}
\label{eq:numerical-hyperbolic-benchmark}
\begin{cases}
\rd X(t)=\left(\mu(t)\sqrt{1+X(t)^2}+\frac12\sigma(t)^2X(t)\right)\rd t
+\sigma(t)\sqrt{1+X(t)^2}\,\rd W(t),\\[0.2em]
X(0)=x_0,
\end{cases}
\end{equation}
with
\[
T=1,
\qquad
x_0=1,
\qquad
\mu(t)=0.2+0.2t^{\gamma_1},
\qquad
\sigma(t)=0.25+0.25t^{\gamma_2},
\]
and \(\gamma_1=0.2\), \(\gamma_2=0.7\).  Applying It\^o's formula to \(Y(t)=\operatorname{arsinh}(X(t))\), we get
\[
\rd Y(t)=\mu(t)\,\rd t+\sigma(t)\,\rd W(t),
\]
and therefore
\begin{equation}
\label{eq:numerical-hyperbolic-exact}
X(T)=\sinh\left(
\operatorname{arsinh}(x_0)+\int_0^T\mu(t)\,\rd t+\int_0^T\sigma(t)\,\rd W(t)
\right).
\end{equation}
The coefficients of \eqref{eq:numerical-hyperbolic-benchmark} are
\[
a(t,x)=\mu(t)\sqrt{1+x^2}+\frac12\sigma(t)^2x,
\qquad
b(t,x)=\sigma(t)\sqrt{1+x^2},
\]
\[
b_y(t,x)=\sigma(t)\frac{x}{\sqrt{1+x^2}},
\qquad
L_1b(t,x)=\sigma(t)^2x.
\]
Thus the diffusion is genuinely nonlinear in the state variable, and the derivative oracle \(b_y\) is nontrivial and state-dependent.  The coefficients still satisfy the assumptions of the input class for a sufficiently large value of \(K\).

Figure~\ref{fig:hyperbolic-benchmark} reports two representative diagnostics for this nonlinear benchmark.  The perturbation response remains essentially linear in the active noise level; the fitted slopes are approximately \(1.001\), \(1.002\), \(1.001\), and \(1.003\) for the drift oracle, diffusion oracle, Wiener path, and all channels, respectively.  In the balanced regime, the fitted rates for \(c=1.2\), \(c=2\), and \(c=3\) are approximately \(0.705\), \(0.695\), and \(0.708\), again consistent with the predicted order \(n^{-\alpha}\).

\begin{figure}[htbp]
\centering
\includegraphics[width=0.98\linewidth]{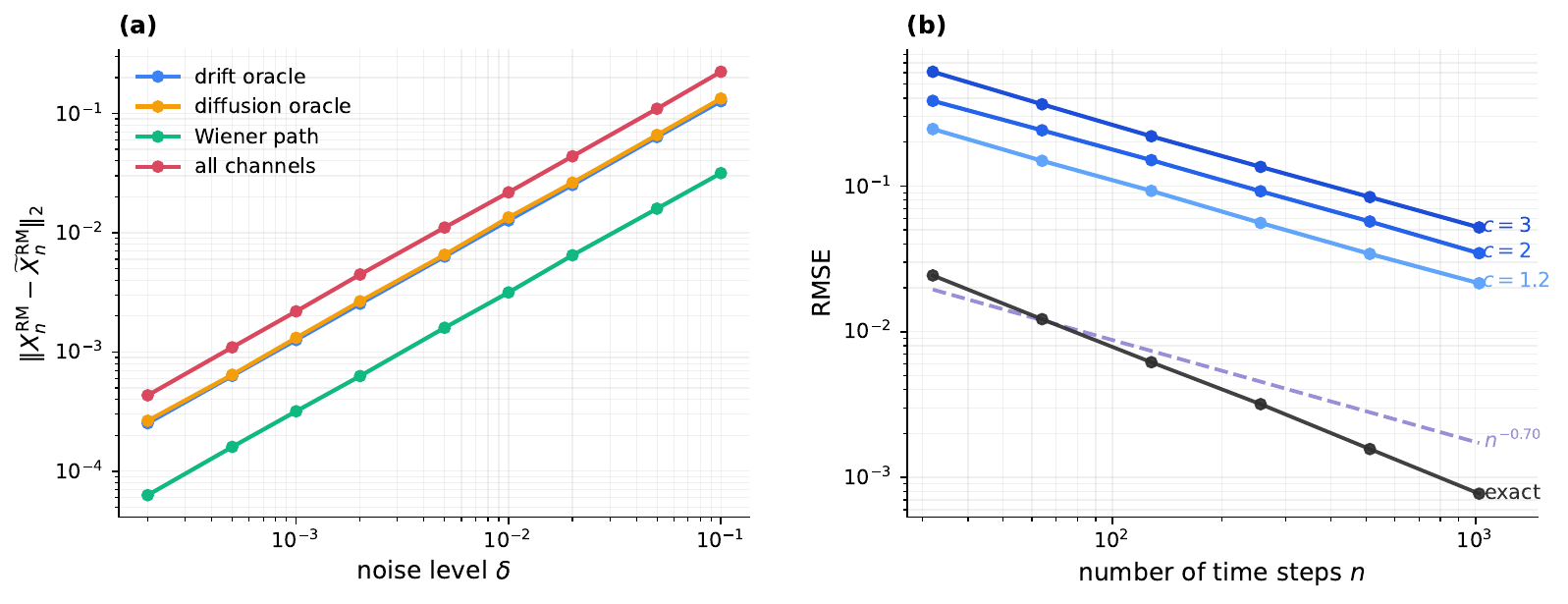}
\caption{Nonlinear exactly solvable benchmark \eqref{eq:numerical-hyperbolic-benchmark}. Left: perturbation response at fixed \(n=1024\). Right: balanced noise regime \(\delta_1(n)=\delta_2(n)=\delta_3(n)=c n^{-\alpha}\). The results show the same qualitative behaviour as for the linear benchmark, now with a genuinely state-dependent derivative oracle \(b_y\).}
\label{fig:hyperbolic-benchmark}
\end{figure}
\FloatBarrier

\subsection{Fingerprint of the Wiener-path lower bound}

Finally, we illustrate the mechanism behind the lower bound proportional to \(\delta_3\).  We use the same simple equation as in the proof of the lower bound,
\[
a\equiv0,
\qquad
b\equiv\beta,
\qquad
\eta=0,
\qquad
\beta=1.
\]
Then
\[
X(T)=\beta W(T).
\]
For this experiment the Wiener-path perturbation is the deterministic Cameron--Martin shift
\[
p_W(t,x)=\frac{t}{\max\{T,1\}},
\qquad
\widetilde W(t)=W(t)+\delta_3p_W(t,W(t)).
\]
For \(T=1\), the induced terminal separation is exactly \(\beta\delta_3\). This experiment is
therefore a sanity check and visualization of the deterministic separation used in the
Cameron--Martin lower-bound construction; see Figure~\ref{fig:delta3-fingerprint}.

\begin{figure}[htbp]
\centering
\includegraphics[width=0.58\linewidth]{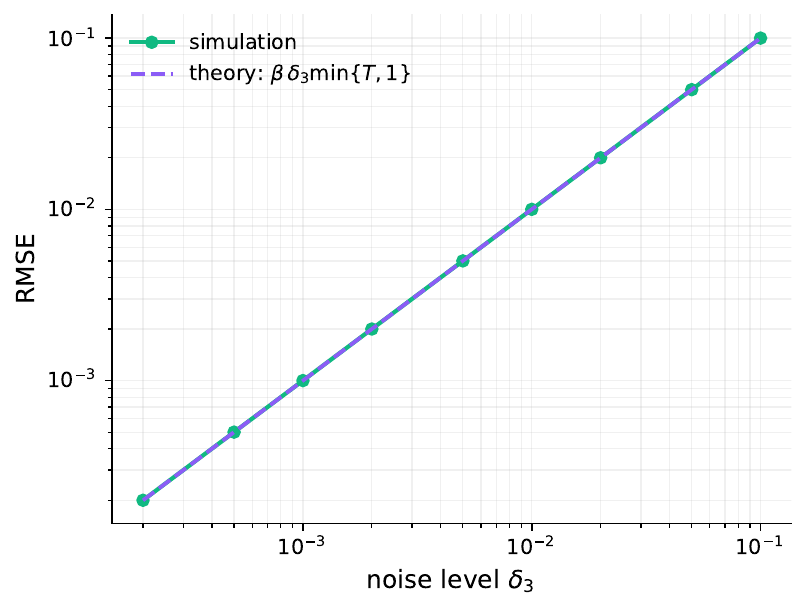}
\caption{Numerical fingerprint of the lower-bound construction for \(\delta_3\). For the test equation \(\rd X(t)=\beta\,\rd W(t)\) and the perturbation \(\widetilde W(t)=W(t)+\delta_3 t/\max\{T,1\}\), the simulated separation reproduces the theoretical value \(\beta\,\delta_3\min\{T,1\}\).}
\label{fig:delta3-fingerprint}
\end{figure}
\FloatBarrier

The fitted slope in this experiment is \(1.000\), and the ratio between the observed RMSE and the value \(\beta\delta_3\min\{T,1\}\) equals one for all tested values of \(\delta_3\), up to numerical precision.  This experiment should be interpreted as a visualization of the statistical indistinguishability construction used in the lower-bound proof.

The experiments are consistent with the four-term structure of the theoretical estimate.
The perturbation-response experiment illustrates the linear dependence on the active
information-noise level. The fixed-noise experiment shows the formation of error floors.
The balanced-regime experiment is consistent with the predicted behaviour under the choice
\(\delta_j(n)\asymp n^{-\alpha}\). Finally, the lower-bound fingerprint visualizes the
Cameron--Martin shift mechanism behind the term linear in \(\delta_3\).

\section{Conclusions}

We have studied the strong approximation of scalar SDEs by a randomized Milstein scheme
under noisy standard information. The information model contains separate noisy channels
for the drift coefficient, the diffusion coefficient, the derivative information required in the
Milstein correction, and a smooth Markovian perturbation of the observed Wiener path.

The main theoretical result is the two-sided minimax estimate
\[
e_N^{(r)}\bigl(\mathcal F(\gamma_1,\gamma_2,K),\delta_1,\delta_2,\delta_3\bigr)
\asymp
N^{-\alpha}+\delta_1+\delta_2+\delta_3,
\qquad
\alpha=\min\{\gamma_1+1/2,\gamma_2\},
\]
for \(r\ge2\), up to multiplicative constants depending only on the parameters of the class.
The upper bound follows by combining the exact-information randomized Milstein rate with
a perturbation stability estimate. The lower bound combines known noisy-coefficient lower
bounds with a new Cameron--Martin two-point argument showing that the contribution of
the Wiener-path perturbation level \(\delta_3\) is unavoidable.

The numerical experiments are consistent with this four-term error structure. On an
exactly solvable benchmark they show an approximately linear response to the perturbation
levels, the formation of error floors for fixed noise levels, and the predicted behaviour in
the balanced regime \(\delta_j(n)\asymp n^{-\alpha}\). The additional test based on the
Cameron--Martin shift visualizes the mechanism behind the lower bound proportional to
\(\delta_3\). These experiments should be interpreted as illustrations of the theoretical
estimates rather than as a separate proof of minimax optimality.

\bf Acknowledgements\rm \\
This research turned into supported by the AGH University of Krakow, Poland under grant no. 16.16.420.054, funded by the Polish Ministry of Science and Higher Education. This research was realized as a part of joint research project between AGH University of Krakow and NVIDIA.



\begin{thebibliography}{22}
\bibitem{BaranekEtAl2026}
M. Baranek, A. Ka{\l}u{\.z}a, P. M. Morkisz, P. Przyby{\l}owicz,
M. Sobieraj,
On the randomized Euler algorithm under inexact information,
{\em J. Comput. Appl. Math.} {\bf 476} (2026), 117070.





\bibitem{CohEl}
S.N. Cohen, R.J. Elliott,
\textit{Stochastic Calculus and Applications}, second ed.,
Birkh\"auser, New York, 2015.
\href{https://doi.org/10.1007/978-1-4939-2867-5}{doi:10.1007/978-1-4939-2867-5}.






 


\bibitem{Hein1}
S. Heinrich,
Lower complexity bounds for parametric stochastic It\^o integration,
in: A.B. Owen, P.W. Glynn (eds.),
{\em Monte Carlo and Quasi-Monte Carlo Methods 2016},
Springer Proceedings in Mathematics \& Statistics, vol.~241,
Springer, Cham, 2018, 295--312.





\bibitem{AKPMPP}
A. Ka{\l}u{\.z}a, P. M. Morkisz, P. Przyby{\l}owicz, Optimal approximation of stochastic integrals in analytic noise model, {\em Appl. Math. and Comput.},  {\bf 356} (2019), 74--91.

\bibitem{Kallenberg}
O. Kallenberg,
{\em Foundations of Modern Probability},
Springer, New York, 1997.

\bibitem{KaratzasShreve}
I. Karatzas, S.E. Shreve, {\em Brownian Motion and Stochastic Calculus, second ed.}, Springer-Verlag, New York, 1991.


\bibitem{KRWU} 
R. Kruse, Y. Wu, A randomized Milstein method for stochastic differential equations with non-differentiable drift coefficients, {\em Discrete Contin. Dyn. Syst. Ser B}, {\bf 24} (2019), 3475--3502.




 \bibitem{PMPP14}
 P. M. Morkisz, P. Przyby{\l}owicz, 
 Strong approximation of solutions of stochastic differential equations with time-irregular coefficients via randomized Euler algorithm, {\em Appl. Numer. Math.} {\bf 78} (2014), 80--94.
 
\bibitem{PMPP17}
 P. M. Morkisz, P. Przyby{\l}owicz, 
Optimal pointwise approximation of SDE's from inexact information, {\em Journal of Computational and Applied Mathematics} {\bf 324} \rm (2017), 85--100.

\bibitem{PMPP19}
 P. M. Morkisz, P. Przyby{\l}owicz, Randomized derivative-free Milstein algorithm for efficient approximation of solutions of SDEs under noisy information,  {\em J. Comput. Appl. Math.} {\bf 383} \rm (2021), 1--22.

\bibitem{NOV} E. Novak, \textit{Deterministic and Stochastic Error Bounds in Numerical Analysis}, Lecture Notes in Mathematics, vol. 1349, New York, Springer--Verlag, 1988. 


\bibitem{Protter}
P. Protter, {\em Stochastic Integration and Differential Equations, second ed.}, Springer-Verlag Berlin Heidelberg, 2005.






\bibitem{TWW88}
  J.F. Traub, G.W. Wasilkowski,  H. Wo\'zniakowski,
  {\em Information-Based Complexity}, Academic Press, New York, 1988.

\bibitem{Tsybakov2009}
A.B. Tsybakov, \textit{Introduction to Nonparametric Estimation},
Springer Series in Statistics, Springer, New York, 2009.

\bibitem{Wer96}
  A.G. Werschulz, The complexity of definite elliptic problems with noisy data. {\em J. Complex.} {\bf 12} (1996), 440--473.
  
\bibitem{Wer97}
  A.G. Werschulz, The complexity of indefinite elliptic problems with noisy data. {\em J. Complex.} {\bf 13} (1997), 457--479.

\end{thebibliography}
\end{document}